\documentclass[12pt]{article}
\usepackage[citecolor=blue,urlcolor=blue]{hyperref}
\usepackage{mathrsfs}
\usepackage[utf8]{inputenc} 
\usepackage[T1]{fontenc}    
\usepackage{url}            
\usepackage{nicefrac}       
\usepackage{amsmath, amsthm, amssymb,amsfonts}
\usepackage[margin=3cm]{geometry}
\newcommand{\N}{\mathbb N}
\renewcommand{\P}{\mathcal P}
\newcommand{\A}{\mathcal A}
\newcommand{\B}{\mathfrak B}

\newcommand{\C}{\mathcal C}
\newcommand{\D}{\mathcal D}

\newcommand{\F}{\mathfrak F}

\newcommand{\I}{\mathscr{I}}
\newcommand{\fC}{\mathscr{C}^{\alpha}}
\newcommand{\bfC}{\mathscr{C}^{\beta}}
\newcommand{\fE}{\mathscr{E}}

\newcommand{\fR}{\mathscr{R}^{\alpha}}
\newcommand{\bfR}{\mathscr{R}^{\beta}}

\renewcommand{\L}{\mathcal L}
\newcommand{\R}{\mathbb R}
\renewcommand{\S}{\mathscr S}
\newcommand{\U}{\mathfrak U}
\newcommand{\V}{\mathfrak V}
\newcommand{\X}{\mathscr X}
\newcommand{\bX}{{\bf{X}}}

\newcommand{\ol}{\overline}
\newcommand{\wt}{\widetilde}
\newcommand{\wh}{\widehat}
\renewcommand{\em}{\itshape}
\newcommand{\PROB}{\Pr}
\newcommand{\btau}{\wt{\tau}}
\newcommand{\bkappa}{\wt{\kappa}}
\newcommand{\bC}{\wt{C}}
\newcommand{\btheta}{\wt{\theta}}
\newcommand{\prodm}{\mu_{\otimes}}
\newcommand{\Z}{\xi}
\newcommand{\ZZ}{\psi}
\newcommand{\bZ}{\wt{\Z}}
\newcommand{\bZZ}{\wt{\ZZ}}
\newtheorem*{theorem*}{Theorem}

\DeclareMathAlphabet{\mathpzc}{OT1}{pzc}{m}{it}

\newcommand{\st}{\text{~such that~}}
\newcommand{\as}{\text{almost surely}}
\newcommand{\norm}[1]{\left\lVert#1\right\rVert} 
\DeclareMathOperator{\var}{Var}

\newtheorem{theorem}{Theorem}

\newtheorem{proposition}{Proposition}
\newtheorem{lemma}{Lemma}
\newtheorem{corollary}{Corollary}
\makeatletter
\addtocounter{footnote}{1} 
\renewcommand\thefootnote{\@fnsymbol\c@footnote}%
\makeatother
 \begin{document}
\title{Inferring the mixing properties of an ergodic process}
\author{Azadeh Khaleghi\thanks{Department of Mathematics \& Statistics, Lancaster University, Lancaster, United Kingdom}  \and  G\'abor Lugosi\thanks{Department of Economics and Business, Pompeu Fabra University;
ICREA, Pg. Llu\'{\i}s Companys 23, 08010 Barcelona, Spain;
Barcelona Graduate School of Economics.}}
\date{}

\maketitle

\begin{abstract}
We propose strongly consistent estimators of the $\ell_1$ norm of the sequence of $\alpha$-mixing (respectively $\beta$-mixing) coefficients of a stationary ergodic process. We further provide strongly consistent estimators of individual $\alpha$-mixing (respectively $\beta$-mixing) coefficients for a subclass of stationary $\alpha$-mixing (respectively $\beta$-mixing) processes with summable sequences of mixing coefficients. The estimators are in turn used to  develop strongly consistent goodness-of-fit hypothesis tests. 
In particular, we develop hypothesis tests to determine whether, under the same summability assumption,  the $\alpha$-mixing (respectively $\beta$-mixing) coefficients of a  process are upper bounded by a given rate function. Moreover, given a sample  generated by a (not necessarily mixing) stationary ergodic process, we provide a consistent test to discern the null hypothesis that the $\ell_1$ norm of the sequence $\boldsymbol{\alpha}$ of $\alpha$-mixing coefficients of the process is bounded by a given threshold $\gamma \in [0,\infty)$ from  the alternative hypothesis that $\norm{\boldsymbol{\alpha}}> \gamma$.  An analogous goodness-of-fit test is proposed for the $\ell_1$ norm of the sequence of $\beta$-mixing coefficients of a stationary ergodic process.
Moreover, the procedure gives rise to an asymptotically consistent test for independence. 
\end{abstract}
\section{Introduction}

  \emph{Mixing} is a  fundamental notion in the theory of stochastic processes.
  Roughly speaking, a stochastic process ${\bf X}=\langle X_t \rangle_{t \in \N}$ indexed by ``time''
  is mixing if events separated by long time intervals are approximately independent. There are various
  notions to quantify such asymptotic independence, including $\alpha$-mixing, $\beta$-mixing, $\phi$-mixing,
  $\rho$-mixing, $\psi$-mixing; see Bradley \cite{BRA07} for a general survey.
  These notions of mixing are particularly useful in time-series analysis where non-asymptotic concentration inequalities
  are available for empirical averages of stationary processes, see, for example, 
Doukhan \cite{DOUK94},
Rio \cite{RIO99},
Bertail, Doukhan, and Soulier \cite{BER06},
Bradley \cite{BRA07},
Bosq \cite{BOS12}. 
In order to be able to take advantage of these tools, in statistical studies one
  often assumes that the process is not only mixing but the mixing coefficients are such that
  the desired concentration inequalities hold.  Despite the widespread use of such assumptions,
   little attention has been paid to validating these conditions. 
  
  In this paper we study
  the problem of estimating $\alpha$ and $\beta$-mixing coefficients. More precisely, we address the following problem.
  Upon observing a finite sample drawn from the trajectory of a real-valued, discrete-time,   stationary and ergodic 
  stochastic process, how can one consistently estimate its $\alpha$ and $\beta$-mixing coefficients?
  Since the $\ell_1$ norm of the sequence of these coefficients plays an important role in quantifying
  how much sample averages differ from their expectations (see, e.g., Lemma \ref{thm:rio_var}),
  we pay special attention to estimating the $\ell_1$ norm. Our main results show that consistent
  estimation of these norms is indeed possible, under the only assumption that the process is stationary 
  ergodic. (If the process is not mixing, the estimators diverge to infinity.) We also show how these estimates can be used
  to derive consistent hypothesis tests on the mixing coefficients. The main difficulty of the estimation problem stems
  from the fact that all mixing coefficients are inherently asymptotic quantities, yet one only has a finite sample
  available. This makes estimation a nontrivial task, especially when no a-priori properties are assumed apart
  from stationarity and ergodicity of the process. 
  
  Our focus in this paper is on $\alpha$ and $\beta$-mixing coefficients, as these are arguably the most
widely used notions of mixing with numerous statistical applications. One may, of course, ask the analogous
  questions on other measures of mixing. A particularly interesting notion is $\phi$ (or uniform) mixing (see, e.g. Bradley \cite{BRA07} for a definition)
  since under such conditions one has Hoeffding-type exponential inequalities in terms of the $\ell_1$ norm
  of the  sequence of $\phi$-mixing coefficients, see Rio \cite[Corollary 2.1]{RIO99} and Samson \cite{SAM00}.
  However, it is unclear if the $\ell_1$ norm of the sequence of $\phi$-mixing coefficients can be consistently estimated.  In this case, a key challenge lies in conditioning on potentially rare events whose probabilities may be arbitrarily small.
  We leave this interesting challenge for future research.

\subsection{\texorpdfstring{$\alpha$}{Lg} and \texorpdfstring{$\beta$}{Lg} mixing} \label{subsec:alphabeta}
We start by defining the notions of mixing relevant to this paper.
Let $(\Omega,\F,\mu)$ be a probability space, and suppose that $\U$ and $\V$ are two $\sigma$-subalgebras  
of $\F$ respectively.
A classical measure of dependence between $\U$ and $\V$, introduced by Rosenblatt \cite{ROS56},  is given by
\begin{align*}
\alpha(\U,\V)&:=\sup_{U \in \U,V \in \V} |\mu(U \cap V) - \mu(U)\mu(V)| 
\end{align*}
which gives rise to the sequence of dependence coefficients $\boldsymbol{\alpha}:=\langle \alpha(m) \rangle_{m \in \N}$ (called the $\alpha$-mixing coefficients)
of a stochastic process  ${\bf X}=\langle X_t \rangle_{t \in \N}$, where  
\begin{align*}
&\alpha(m):=\sup_{j \in \N}\alpha(\sigma(\{X_{t}: 1\leq t \leq j\}),\sigma(\{X_{t}: t \geq j+m\}))~.
\end{align*}
The $\beta$-dependence $\beta(\U,\V)$ between $\U$ and $\V$ was originally introduced by Volkonskii and Rozanov \cite{VoRo59,VoRo61} as follows (see also Rio \cite{RIO99}). Let $\iota(\omega) \mapsto (\omega,\omega)$ be the injection map from $(\Omega,\F)$ to  $(\Omega \times \Omega,\U \otimes \V)$, where $\U \otimes \V$ is the product sigma algebra generated by $\U \times \V$.
 Let $\prodm$ be the probability measure defined on  $(\Omega \times \Omega,\U \otimes \V)$ obtained as the pushforward measure of  $\mu$ under $\iota$. Let $\mu_{\U}$ and $\mu_{\V}$ denote the restriction of $\mu$ to $\U$ and $\V$ respectively. Then
\begin{align*}
\beta(\U,\V)&:=\sup_{W \in \sigma(\U \times \V)} |\prodm(W) - \mu_{\U}\times \mu_{\V}(W)|~,
\end{align*}
where $\mu_{\U}\times \mu_{\V}$ is the product measure on $(\Omega \times \Omega,\U \otimes \V)$ obtained from $\mu_{\U}$ and $\mu_{\V}$. This measure of dependence leads to the sequence $\boldsymbol{\beta}:=\langle\beta(m)\rangle_{m \in \N}$ of $\beta$-mixing coefficients of a stochastic process ${\bf X}$, where 
\begin{equation*}
\beta(m):=\sup_{j \in \N}\beta(\sigma(\{X_{t}: 1\leq t \leq j\}),\sigma(\{X_{t}: t \geq j+m\})).
\end{equation*}

A stochastic process is said to be $\alpha$-mixing or {\em strongly mixing}, if it exhibits an asymptotic independence of the form $\lim_{m\rightarrow\infty}\alpha(m)=0$. It is said to be $\beta$-mixing or {\em absolutely regular} if $\lim_{m \rightarrow \infty}\beta(m)=0$. It is straightforward to check that $\beta(\U,\V) \geq \alpha(\U,\V)$ so that absolute regularity implies strong mixing. 
Moreover,  summability of $\boldsymbol\alpha$ (respectively $\boldsymbol{\beta}$) is clearly a sufficient but not necessary condition for a process to be $\alpha$-mixing (respectively $\beta$-mixing).

\subsection{The estimation problem}\label{subsec:estimation}  
Regardless of whether or not a process is mixing, its mixing coefficients provide a measure of its dependence structure. Moreover, as mentioned above, individual $\alpha$-mixing (respectively $\beta$-mixing) coefficients and/or their sum $\norm{\boldsymbol{\alpha}}=\sum_{m \in \N}\alpha(m)$ (respectively  $\norm{\boldsymbol{\beta}}=\sum_{m \in \N}\beta(m)$) commonly appear in concentration inequalities for dependent processes.
Thus, in order to use these bounds in a statistical problem where the samples may be dependent, knowledge of the sequences $\boldsymbol{\alpha}$, $\boldsymbol{\beta}$ or at least of their $\ell_1$ norms is required. 

For certain subclasses of dependent processes, bounds on the mixing coefficients are known. For example, conditions for the geometric ergodicity of Markov chains have been well studied, see Meyn and Tweedie \cite{MEY12} and references therein. More recently, some upper bounds on the mixing rates of non-stationary ARCH processes were proposed by Fryzlewicz and Rao \cite{FRY11}. 
However, for larger classes of stationary processes, the mixing
coefficients are typically unknown, and surprisingly little research
has been devoted to the problem of estimating mixing coefficients. One
exception is the work of McDonald, Shalizi, and Schervish
\cite{MCD11,MCD15} who provide estimators for $\beta$-mixing
coefficients and show consistency of their estimators. Unfortunately,
due to a lack of precision in the presentation, we were unable to
verify some of the main claims of these papers.  Our work does not
build upon these results and we consider a more general setting where,
apart from ergodicity, no assumptions are required on the process
distributions (as opposed to some implicit assumptions on the
existence and smoothness of finite-dimensional densities made in
\cite{MCD11,MCD15}).

Towards an adaptive approach, given a sample generated by a stationary ergodic process,  our objective in this paper is to estimate its dependence structure as reflected by its $\alpha$-mixing and $\beta$-mixing coefficients. 
To this end, in Section~\ref{sec:est_a}, we first focus on a subclass of stationary ergodic processes  which are $\alpha$-mixing with the additional property that their sequences of  $\alpha$-mixing coefficients are summable. For this class, without knowledge of any upper-bounds on $\norm{\boldsymbol{\alpha}}$  and merely using the fact that $\norm{\boldsymbol{\alpha}}<\infty$, 
we provide asymptotically consistent estimators of the individual $\alpha$-mixing coefficients as well as for $\norm{\boldsymbol{\alpha}}$. The consistency results for these estimators are established via Theorem~\ref{thm:azal_as1} and  Theorem~\ref{thm:azal_as2} respectively. We rely on Rio's covariance inequality \cite[Corollary 1.1]{RIO99} to control the variance of partial sums. 
Next, we propose an alternative approach for consistently estimating $\norm{\boldsymbol{\alpha}}$ of a (not necessarily mixing) stationary ergodic process where we no longer require $\boldsymbol{\alpha}$ to be summable.
In this case, if the process happens to be $\alpha$-mixing with $\norm{\boldsymbol{\alpha}}<\infty$, then the proposed estimator converges to $\norm{\boldsymbol{\alpha}}$, otherwise the estimator diverges to infinity.
The weak and strong consistency properties of these estimators follow from Theorems~\ref{thm:whp} and \ref{thm:exhaustive_strong}. The results in this section rely on Lemma~\ref{lem:dynk} and Proposition~\ref{prop:phik} which together show that the  approximation given by \eqref{eq:approx} which is based on the cylinder sets converges to $\alpha(m)$.
 
In Section~\ref{sec:est_b} we provide analogous results for the sequence of $\beta$-mixing coefficients of a stationary ergodic process. Most of the arguments are similar to those given in Section~\ref{sec:est_a}, and in particular since $\alpha(m)\leq \beta(m)$, the same covariance inequality is sufficient to control the variance of partial sums in the estimation. However, a key challenge in this case  is that unlike $\alpha(\U,\V)$, the $\beta$-dependence between $\U$ and $\V$  is defined for the product space $(\Omega \times \Omega, \U \otimes \V, \prodm)$. In order to propose an approximation  that can be estimated from a single sample-path, we rely on the identity,
\begin{equation*}
\beta(\U,\V)=\sup \frac{1}{2}\sum_{i \in \N}\sum_{j\in \N}|\mu(U_i\cap V_j)-\mu(U_i)\mu(V_j)|
\end{equation*}
where the supremum is taken over all pairs of countable partitions $\{U_1,U_2,U_3,\ldots\}$ and $\{V_1,V_2,V_3,\ldots\}$ of $\Omega$ such that $U_i \in \U$ and $V_j \in \V$ for each $i,~j\in \N$; see Bradley \cite[Vol. 1, P. 67, Note 2]{BRA07}. This, together with \cite[Vol. 1 Proposition 3.21]{BRA07} gives rise to Lemma~\ref{lem:dynk_b} which in turn leads to Proposition~\ref{prop:phik_b}. These results are analogues of Lemma~\ref{lem:dynk}  and Proposition~\ref{prop:phik}, and show that the approximation \eqref{eq:approx_b} that is based on the cylinder sets converges to $\beta(m)$. 

\subsection{Hypothesis testing}  
In Section~\ref{sec:gf} we show how our estimators can be used to construct goodness-of-fit tests.  First, given a sample generated by a stationary $\alpha$-mixing (respectively $\beta$-mixing) process $\mu$, we use our estimator of $\alpha(m)$ (respectively $\beta(m)$) to test the null hypothesis that the $\alpha$-mixing (respectively $\beta$-mixing) coefficients of $\mu$ are bounded by a given {\em rate function} $\gamma: \N \mapsto [0,1]$ against the alternative hypothesis $H_1$ that there exists some $m \in \N$ such that $\alpha(m)>\gamma(m)$ (respectively $\beta(m)>\gamma(m)$). The consistency of these tests follow from Theorems~\ref{thm:gamma_rate} and \ref{thm:gamma_rate_b}, provided that the sequence of $\alpha$-mixing (respectively $\beta$-mixing) coefficients of the process is summable. 
Interestingly, Nobel \cite{NOB06} used hypothesis testing to estimate polynomial decay rates for covariance-based mixing conditions. By contrast, we do not require the rate functions   to be polynomial or to belong to any specific function class. Moreover, given a sample  generated by a (not necessarily mixing) stationary ergodic process $\mu$, we construct tests to discern the null hypothesis $H_0$ that $\norm{\boldsymbol{\alpha}}$ (respectively $\norm{\boldsymbol{\beta}}$) is bounded by a given threshold $\gamma \in [0,\infty)$ from  the alternative hypothesis $H_1$ that it exceeds $\gamma$. The consistency of these procedures follow from Theorem~\ref{thm:theta_rate} (for $\alpha$-mixing coefficients) and Theorem~\ref{thm:theta_rate_b} (for $\beta$-mixing coefficients). 
As a direct consequence, we obtain strongly consistent tests for independence, obtaining an alternative proof of the main result of Morvai and Weiss \cite{MOR11}; this is stated as Corollary~\ref{cor:iidt2}.

\section{Preliminaries}
\renewcommand{\rho}{\varrho}
In this section we fix our notation and introduce some basic definitions. 
Let $\X$ together with its Borel $\sigma$-algebra $\B_{\X}$ be a measurable space. In order to keep the notation simple and to avoid uninteresting technicalities, we take $\X=[0,1]$; but we would like to point out that extensions to more general spaces including to $\R$ and $\R^d$ are straightforward.
Denote by $\B^{(k)}$ the product $\sigma$-algebra on $\X^k,~k\in \N$.
Let $\X^{\N}$ be the set of all $\X$-valued infinite sequences indexed by $\N$. 
A (discrete-time) stochastic process is a probability measure $\mu$ on the space $(\X^{\N}, \F)$ where $\F$ denotes the 
Borel $\sigma$-algebra on $\X^{\N}$ generated by the cylinder sets.  
Associated with the process is a sequence of random variables $\bX:=\langle X_t \rangle_{t\in \N}$ where $X_t:\X^{\N} \rightarrow \X$ are coordinate projections such that $X_t(\boldsymbol{a})=a_t$ for $\boldsymbol{a}=\langle a_t \rangle_{t \in \N} \in \X^{\N}$.  We use the term {\em (stochastic) process}  to refer to either the measure $\mu$ or to its corresponding sequence of random variables $\bX$; the distinction should be clear from the context.
A process is $\bX$ is stationary if for all $i,k \in \N$ and all $B \in \B^{(k)}$ we have 
\begin{equation*} 
\Pr((X_1,\ldots,X_k) \in B)=\Pr((X_{1+i},\ldots,X_{k+i}) \in B).
\end{equation*} 
Let $S:\X^\N \rightarrow \X^\N$ denote the (left) shift transformation on $\X^{\N}$ which maps $\boldsymbol{a}:=(a_1,a_2,\ldots) \in \X^\N$ to $S\boldsymbol{a}=(a_2,a_3,\ldots)$. It is continuous relative to the product topology on $\X^{\N}$ and defines 
a set transformation $S^{-1}$ given by $S^{-1}A:=\{\boldsymbol{a} \in \X^{\N}:S\boldsymbol{a} \in A\},~ A \subseteq \X^\N$. Thus, it is straightforward to check that $S^{-1}$ is Borel measurable and that  
stationarity of $\bX$ translates to the condition that $\mu(S^{-1} B)=\mu(B)$ for all $B \in \F$. A stationary process is ergodic if every shift-invariant measurable set has measure $0$ or $1$ so that if $S^{-1}B=B$ for some $B \in \F$ then $\mu(B) \in \{0,1\}$. Recalling the definition of a strongly mixing process given in the introduction, note the well-known fact that a stationary  $\alpha$-mixing process is ergodic, see, for example, Bradley \cite{BRA07}.

\section{Estimation}\label{sec:est}
Consider a stationary ergodic  process $\mu$ with corresponding sequence of random variables $\bX=\langle X_t\rangle_{t \in \N}$ and sequences of $\alpha$-mixing and $\beta$-mixing coefficients  $\boldsymbol{\alpha}:=\langle \alpha(m) \rangle_{m \in \N}$ and $\boldsymbol{\beta}:=\langle \beta(m) \rangle_{m \in \N}$ respectively. In this section we introduce estimators of $\norm{\boldsymbol{\alpha}}$ and $\alpha(m)$, as well as of  $\norm{\boldsymbol{\beta}}$ and $\beta(m)$ -
  and establish their consistency -- these are the main results of the paper. 
  
We consider the,
  somewhat simpler,  case of $\alpha$-mixing coefficients in Section \ref{sec:est_a}.  The corresponding estimators for $\beta$-mixing coefficients are given Section \ref{sec:est_b}. 
 For the most part, the estimators -- and proofs -- concerning the $\beta$-mixing coefficients are analogous to their $\alpha$-mixing counterparts. However, as discussed in Sections~\ref{subsec:alphabeta} and \ref{subsec:estimation}, there is a subtle distinction between the measurable spaces on which the two mixing coefficients are defined. This  calls for a slightly different treatment in the case of $\beta$-mixing coefficients, giving rise to a number of technical challenges which we address as part of our analysis in  Section \ref{sec:est_b}.

The following notation will be used throughout the section.
Let $\Delta_{k,\ell}$ 
be the set of dyadic cubes in $\X^k,~k\in \N$ of side-length $2^{-\ell}$. 
That is, 
\[
  \Delta_{k,\ell}:=\left\{ \left[\frac{i_1}{2^{\ell}}, \frac{i_1+1}{2^{\ell}}\right)\times \ldots \times \left[\frac{i_k}{2^{\ell}}, \frac{i_k+1}{2^{\ell}}\right): i_j \in \{0,\ldots,2^{\ell}-1\},~j \in \{1,\ldots,k\}\right\}
\]
For each $k,\ell \in \N$ we denote by $
\D_{k,\ell}:=\P(\Delta_{k,\ell})
$  the power-set of $\Delta_{k,\ell}$. 
For a given $B \in \B^{(k)},~k \in \N$, we denote the event $\{X_i,\ldots,X_{i+k-1} \in B\},~i \in \N$ by $[B]_{i}^{i+k}$. For $m, \ell \in \N,~n >m \in \N$ and  each $j \in \{1,\ldots,n-m\}$  define the $\sigma$ algebra generated by the sets $[A]_{1}^{j}$ for
$A \in \D_{j,\ell}$ by
\begin{equation}\label{eq:defn_Fj}
\F_{1}^{j}(\ell):=\sigma\left (\left \{[A]_{1}^{j}: A \in \D_{j,\ell}\right \}\right)
\end{equation}
and similarly let 
\begin{equation}\label{eq:defn_Fj+m}
\F_{j+m}^{n}(\ell):=\sigma \left (\left \{ [B]_{j+m+1}^{n}: B\in \D_{j',\ell},~\text{with}~j':=n-m-j\right\}\right)~.
\end{equation}

\subsection{Estimating \texorpdfstring{$\alpha(m)$}{Lg} and \texorpdfstring{$\norm{\boldsymbol{\alpha}}$}{Lg}}\label{sec:est_a}
Since the mixing coefficients $\alpha(m)$ of ${\bf X}$ are inherently asymptotic quantities, we first approximate
  them by quantities that only depend on finite-dimensional projections of the process and therefore they can be estimated
  from finite samples of the process. To this end, for each $m,\ell \in \N,~n > m \in \N$ and $j \in \{1,\ldots,n-m\}$, define
\begin{equation}\label{eq:approx}
\alpha_{n,j}^{\ell}(m):= \sup_{\substack{A \in \D_{j,\ell} \\ B \in \D_{j',\ell}}} \left |\mu \left ([A]_{1}^{j} \cap [B]_{j+m+1}^{n} \right ) -\mu \left( [A]_{1}^{j} \right ) \mu \left( [B]_{j+m+1}^{n}\right )  \right |
\end{equation}
where $j':=n-m-j+1$. 
In Lemma~\ref{lem:dynk} below we show that to approximate the $\alpha$-dependence between $\F_{1}^{j}(\ell)$ and $\F_{j+m}^{n}(\ell)$, it suffices to use $\alpha_{n,j}^{\ell}(m)$ where the supremum is taken over the smaller classes of sets $\D_{j,\ell}$ and $\D_{j',\ell},~\text{with}~j':=n-m-j$.  
\begin{lemma}\label{lem:dynk}
For $m, \ell \in \N,~n >m \in \N$  and $j \in \{1,\ldots,n-m\}$ we have 
\[\alpha_{n,j}^{\ell}(m)=\alpha( \F_{1}^{j}(\ell), \F_{j+m}^{n}(\ell)).\] 
\end{lemma}
\begin{proof}[Proof of Lemma~\ref{lem:dynk}]
Fix $m, \ell \in \N,~n >m \in \N$ and some $j \in \{1,\ldots,n-m\}$. First, it is clear that 
$\alpha_{n,j}^{\ell}(m) \leq \alpha( \F_{1}^{j}(\ell), \F_{j+m}^{n}(\ell))$. 
Therefore, it suffices to show that $\alpha_{n,j}^{\ell}(m) \geq \alpha( \F_{1}^{j}(\ell), \F_{j+m}^{n}(\ell))$. 
Define
\begin{align*}
&\C_{1}^{j}(\ell):=\left \{[A]_{1}^{j}: A \in \D_{j,\ell} \right \}
\end{align*}
and 
\begin{align*}
&\C_{j+m}^{n}(\ell):=\left \{ [B]_{j+m+1}^{n}: B\in \D_{j',\ell},~j':=n-m-j\right\}.
\end{align*}
Let $\A_{1}^{j}(\ell)$ and $\A_{j+m}^{n}(\ell)$ correspond to algebras over $\C_{1}^{j}(\ell)$ and $\C_{j+m}^{n}(\ell)$
 respectively. 
It trivially holds that $\A_{1}^{j}(\ell)$ forms a $\pi$-system, since by definition, it is non-empty and closed under finite intersections.
On the other hand we have
\begin{equation}\label{eq:alpha_alg-}
\sup_{\substack{U \in \A_{1}^{j}(\ell) \\ V \in \C_{j+m}^{n}(\ell) }}\left | \mu(U \cap V)-\mu(U)\mu(V) \right|\leq \alpha_{n,j}^{\ell}(m)~.
\end{equation}
To see this, first note that
by definition of $\alpha_{n,j}^{\ell}(m)$, we have
\begin{equation}\label{eq:alpha_bound_assmp}
\sup_{\substack{U \in \C_{1}^{j}(\ell) \\ V \in \C_{j+m}^{n}(\ell) }}\left | \mu(U \cap V)-\mu(U)\mu(V) \right|= \alpha_{n,j}^{\ell}(m)~.
\end{equation}
Moreover, by construction $\C_{1}^{j}(\ell)$ is already closed under finite unions and intersections. 
Therefore, by \eqref{eq:alpha_bound_assmp} for any $k \in 1,\ldots, |\C_{1}^{j}(\ell)|$ and any $U_1,\ldots,U_k \in \C_{1}^{j}(\ell)$ we have
\begin{align*}
\max\left\{\sup_{V \in \C_{j+m}^{n}(\ell)}\left | \mu(U \cap V)-\mu(U)\mu(V) \right|, \sup_{V \in \C_{j+m}^{n}(\ell)}\left | \mu(U' \cap V)-\mu(U')\mu(V) \right|\right\}\leq \alpha_{n,j}^{\ell}(m)
\end{align*}
where $U :=\bigcup_{i=1}^k U_i$ and $U':=\bigcap_{i=1}^k U_i$. Moreover, for any $U  \in \C_{1}^{j}(\ell)$ and any $V \in  \C_{j+m}^{n}(\ell)$, we have 
\begin{align*}
\left | \mu(U^c \cap V)-\mu(U^c)\mu(V) \right| 
&=\left | \mu(V)-\mu(U\cap V )-(1-\mu(U))\mu(V) \right| \\
&=\left |\mu(U)\mu(V)-\mu(U\cap V )  \right| \leq \alpha_{n,j}^{\ell}(m)~.
\end{align*} 
Let $\L$ be the largest algebra of subsets of $\F$ such that
\begin{align}
\sup_{\substack{U \in \L\\ V \in\C_{j+m}^{n}(\ell)}}\left | \mu(U \cap V)-\mu(U)\mu(V) \right|\leq \alpha_{n,j}^{\ell}(m)~.
\end{align}
It is straightforward to verify that $\L$ forms a $\lambda$-system. 
To see this, first note that since $\alpha_{n,j}^{\ell}(m) \geq 0$,  it clearly holds that $\X^\N \in \L$. 
Next, take $U_1 \subseteq U_2 \in \L$. 
Since $\mathcal L$ is an algebra, it is closed under complementation as well as under pairwise unions and intersections. 
Therefore, $(U_2\setminus U_1) \in \mathcal L$. Finally, consider a countable sequence of increasing subsets $U_{i} \subseteq U_{i+1}  \in \L,~i \in \N$ and define $U :=\bigcup_{i=1}^{\infty} U_i$. It follows from the continuity of probability measure that  $U \in \L$. 
More specifically, for $V \in \C_{j+m}^{n}(\ell)$ define $\ol{U}_i:=U_i \cap V,~i \in \N$, and let $\ol{U}:=\bigcup_{i=1}^{\infty}\ol{U}_i$.  Observe that $\ol{U}_i \subseteq \ol{U}_{i+1}$ so that $\lim_{n\to\infty}\mu(\ol{U}_n)=\mu(\ol{U})$. Similarly, it holds that $\lim_{n\to \infty} \mu(U_n)=\mu(U)$. 
Fix $\epsilon >0$ there exist $N_{\epsilon}, N'_{\epsilon}$ such that for all $n \geq \max\{N_{\epsilon}, N'_{\epsilon}\}$ we have 
$|\mu(U_n)-\mu(U)|\leq \epsilon$ and $|\mu(\ol{U}_n)-\mu(\ol{U})|\leq \epsilon$. 
Therefore, 
\begin{align*}
\Big | \mu(U \cap V)-\mu(U)\mu(V) \Big|
&=
\left |\mu(\ol{U})-\mu(U)\mu(V)\right |\\
&\leq \left |\mu(\ol{U}_n)-\mu(U_n)\mu(V)\right |+2\epsilon \\
&\leq \alpha_{n,j}^{\ell}(m)+2\epsilon~.
\end{align*}
Since the choice of epsilon is arbitrary, it follows that $U \in \L$. 
Therefore $\L$ is a $\lambda$-system. Moreover, by \eqref{eq:alpha_alg-} we have $\A_{1}^{j}(\ell) \subseteq \L$. 
Thus, as follows from Dynkin's $\pi-\lambda$ theorem, we can deduce that $\sigma(\A_{1}^{j}(\ell)) \subseteq \L$ so that
\begin{equation}\label{eq:dynk_L}
\sup_{\substack{U \in \sigma(\A_{1}^{j}(\ell)) \\ V \in\C_{j+m}^{n}(\ell)}}\left | \mu(U \cap V)-\mu(U)\mu(V) \right|\leq \alpha_{n,j}^{\ell}(m)~.
\end{equation} 
On the other hand, let $\L'$ be the largest algebra of subsets of $\F$ such that 
\begin{equation}
\sup_{\substack{U \in \sigma(\A_{1}^{j}(\ell)) \\ V \in\L'}}\left | \mu(U \cap V)-\mu(U)\mu(V) \right|\leq \alpha_{n,j}^{\ell}(m)~.
\end{equation}
In much the same way as with $\L$, it is easy to see that $\L'$ is a $\lambda$-system. 
Moreover, by \eqref{eq:dynk_L} and an argument analogous to that concerning $\C_{1}^{j}(\ell)$, we 
can conclude that the algebra of subsets of $\C_{j+m}^{n}(\ell)$ is included in $\L'$, that is, $\A_{j+m}^{n}(\ell) \subseteq \L'$.
Hence, by another application of the $\pi-\lambda$ theorem,
we obtain $\sigma(\A_{j+m}^{n}(\ell))\subseteq \L'$. This leads to 
\begin{equation}
\sup_{\substack{U \in \sigma(\A_{1}^{j}(\ell))\\ V \in \sigma(\A_{j+m}^{n}(\ell))}}\left | \mu(U \cap V)-\mu(U)\mu(V) \right|\leq \alpha_{n,j}^{\ell}(m)~.
\end{equation}
Observing that $\sigma(\A_{1}^{j}(\ell))=\F_{1}^j(\ell)$ and $\sigma(\A_{j+m}^{j}(\ell))=\F_{j+m}^n(\ell)$, we obtain
\begin{equation}
\alpha(\F_{1}^j(\ell),\F_{j+m}^n(\ell)) \leq \alpha_{n,j}^{\ell}(m)
\end{equation}
and the result follows.
\end{proof}
For $m, \ell \in \N,~n >m \in \N$ let 
\begin{equation}
\alpha_n^{\ell}(m)=\max_{j \in \{1,\ldots,n-m\}}\alpha_{n,j}^{\ell}(m)~.
\end{equation}
The next proposition shows that indeed, $\alpha_n^{\ell}(m)$ approximates 
$\alpha(m)$ for sufficiently large values of $\ell$ and $n$.
\begin{proposition}\label{prop:phik}
For every $m\in \N$ we have
$
\displaystyle \lim_{ n,\ell \rightarrow \infty}\alpha^{\ell}_n(m)=\alpha(m)~.
$
\end{proposition}
\begin{proof}[Proof of Proposition~\ref{prop:phik}]
Fix  $m,\ell \in \N$ and $n >m$.
Observe that for each $j \in \{1,\ldots,n-m\}$ we have $\F_{1}^{j}(\ell) \subset \F_{1}^{j}(\ell+1),~\ell \in \N$ and $\F_{j+m}^n(\ell)\subset \F_{j+m}^{n+1}(\ell+1),~n,\ell \in \N$ are each a sequence of $\sigma$-algebras with $\bigvee_{\ell}^{\infty}\F_{1}^{j}(\ell) =\sigma(X_{1},\ldots,X_j)$ and 
$\bigvee_{n,\ell}^{\infty}\F_{j+m}^{n}(\ell)=\sigma(\{X_{t}: t \geq j+m+1\})$. 
Therefore, by Bradley \cite[Vol. 1, Proposition 3.18]{BRA07} we obtain, 
\begin{equation}\label{eq:prop318} 
\lim_{n,\ell \to \infty}\alpha(\F_{1}^{j}(\ell),\F_{j+m}^n(\ell))=\alpha(\sigma(\{X_t: t \in \{1,\ldots, j\}\}),\sigma(\{X_{t}: t \geq j+m+1\}))~.
\end{equation}
It is straightforward to  check that for each $m \in \N$ we have
\begin{align}\label{eq:supmax}
\sup_{n,\ell}\max_{j \in \{1,\ldots,n-m\}}\alpha_{n,j}^{\ell}(m)=\sup_{j \in \N}\sup_{n\geq j+m+1} \sup_{\ell}\alpha_{n,j}^{\ell}(m)~.
\end{align}
To see this 
first let $c:= \sup_{j \in \N}\sup_{n\geq j+m} \sup_{\ell}\alpha_{n,j}^{\ell}(m)$ and 
fix some $\epsilon>0$; by the definition of $\sup$ there exist $\ell^*, j^* \in \N$ and  $n^* \geq j^*+m$ such that  $\alpha_{n^*,j^*}^{\ell}(m) \geq c-\epsilon$. We have $c-\epsilon \leq   \alpha_{n^*,j^*}^{\ell}(m) \leq \sup_{n,\ell}\max_{j \in \{1,\ldots,n-m\}} \alpha_{n,j}^{\ell}(m)$. Similarly, let $c':= \sup_{n,\ell}\max_{j \in \{1,\ldots,n-m\}}\alpha_{n,j}^{\ell}(m)$, and note that there exist some $n',\ell'  \in \N$ such that $\max_{j \in 1,\ldots,n'-m}\alpha_{n',j}^{\ell'}(m) \geq c'-\epsilon$. Hence, 
\[c'-\epsilon \leq \max_{j \in 1,\ldots,n'-m}\alpha_{n',j}^{\ell'}(m)  \leq \sup_{j \in \N}\sup_{n\geq j+m} \sup_{\ell}\alpha_{n,j}^{\ell}(m).\] Since the choice of $\epsilon$ is arbitrary, \eqref{eq:supmax} follows. 
We obtain
\begin{align}
\lim_{n,\ell \rightarrow\infty} \alpha_{n}^{\ell}(m) 
&= \lim_{n,\ell \rightarrow \infty}\max_{j \in \{1,\ldots,n-m\}}\alpha_{n,j}^{\ell}(m)\nonumber \\
&=\sup_{n,\ell}\max_{j \in \{1,\ldots,n-m\}}\alpha_{n,j}^{\ell}(m)\label{eq:limissup1}\\
&=\sup_{j \in \N}\sup_{ n \geq j+m}\sup_{\ell}\alpha_{n,j}^{\ell}(m)\label{eq:supmax_}\\
&=\sup_{j \in \N}\lim_{n,\ell \rightarrow \infty}\alpha_{n,j}^{\ell}(m) \label{eq:limissup2}\\
&=\sup_{j \in \N}\lim_{n,\ell \rightarrow \infty}\alpha(\F_{1}^{j}(\ell),\F_{j+m}^n(\ell)) \label{eq:approxissigmaalg}\\
&=\sup_{j \in \N}\alpha(\sigma(\{X_t:t \in \{1,\ldots, j\}\}),\sigma(\{X_t:t \geq j+m+1\}))\label{eq:fromprop318}\\
&=\alpha(m)~,\nonumber
\end{align}
where \eqref{eq:limissup1} and   \eqref{eq:limissup2} follow from the fact that  for a fixed $m \in \N$, $\alpha_{n,j}^{\ell}(m)$ is an increasing function of $n,\ell$, \eqref{eq:supmax_} follows from \eqref{eq:supmax}, \eqref{eq:approxissigmaalg} follows from Lemma~\ref{lem:dynk}, and \eqref{eq:fromprop318} follows from \eqref{eq:prop318}.
\end{proof}

Now we are ready to introduce the natural empirical estimates of the approximate mixing coefficients $\alpha^{\ell}_n(m)$. The key ingredient of the analysis is the concentration inequality of Lemma \ref{thm:rio_var}
below.

For $t \in \N$  
define the empirical measure $\mu_t({\bf X},\cdot):  \B^{(k)} \rightarrow [0,1],~k \in \N$ as
\begin{equation}\label{eq:emp}
\mu_t({\bf X},B):=\frac{1}{t}\sum_{i=0}^{t-1} \chi_{B}\{X_{ik+1},\ldots, X_{(i+1)k}\}~,
\end{equation}
where $\chi$ is the indicator function. Lemma~\ref{lem:empf} provides a simple concentration bound on the empirical measure of a cylinder set $[D]_1^k$ for any $D \in \D_{k,\ell},~k,\ell \in \N$. The proof relies on the following variance bound of Rio:
\begin{lemma}[{Rio \cite[Corollary 1.1]{RIO99}}]\label{thm:rio_var}
Let $\langle Y_i \rangle_{i \geq 0}$ be a stationary sequence of $[-1,1]$-valued random variable. Define $\ol{\alpha}(m):=\alpha(\sigma(Y_0),\sigma(Y_m))$ and suppose that $\norm{\ol{\boldsymbol{\alpha}}}:=\sum_{m \in \N} \ol{ \alpha}(m) <\infty$. For each $t \in \N$, let $S_t:=Y_0+\ldots+Y_{t-1}$. We have
\begin{equation*}
\var(S_t) \leq 4t\norm{\ol{\boldsymbol{\alpha}}}.
\end{equation*}
\end{lemma}
\begin{lemma}\label{lem:empf}
For all $ D \in \D_{k,\ell},~k,\ell, t \in \N$ and for every $\epsilon >0$, we have
\begin{equation*}
\Pr\left ( \left |\mu_t({\bf X},D)-\mu([D]_{1}^{k}) \right| \geq \epsilon \right)\leq \frac{4\norm{\boldsymbol{\alpha}}}{t\epsilon^2}~.
\end{equation*}
\end{lemma}
\begin{proof}[Proof of Lemma~\ref{lem:empf}]
Let $\ol{\N}:=\N\cup\{0\}$. 
For each $D \in \D_{k,\ell},~k,\ell \in \N$ define the $[-1,1]$-valued sequence of random variables $\langle Y_i\rangle_{i \in \ol{\N}}$ where
\begin{equation}
Y_i:=\chi_D\{X_{i k+1},\ldots,X_{(i+1)k}\}-\mu([D]_{1}^{k}),~i \in \ol{\N}
\end{equation}
Observe that $\langle Y_i \rangle_{ti\in \ol{\N}}$ is a zero-mean $[-1,1]$-valued stationary process and for each $m \in \N$ we have,
\begin{align*}
\alpha(\sigma(Y_{0}),\sigma(Y_{m}))&\leq \alpha(\sigma(X_1,\ldots,X_k),\sigma(X_{mk+1},\ldots,X_{(m+1)k}))\\ & \leq \alpha(k (m-1)+1)
\end{align*}
where the first inequality follows from the fact that by definition we have $\sigma(Y_{0}) \subset \sigma(X_1,\ldots,X_k)$ and  that $\sigma(Y_{m}) \subset \sigma(X_{mk+1},\ldots,X_{(m+1)k})$ and the second inequality follows form the definition of the $\alpha$-mixing coefficients of ${\bf X}$.
Thus, we obtain
$$\sum_{m \in \N} \alpha(\sigma(Y_{0}),\sigma(Y_{m})) \leq \sum_{m \in \N} \alpha(k(m-1)+1) \leq \sum_{m\in\N} \alpha(m) = \norm{\boldsymbol{\alpha}}$$
where, the last inequality follows from the fact that $\alpha(m)$ is a decreasing sequence so that 
for every $u>v \in \N$ we have $\alpha(u)\leq \alpha(v)$. 
Let $S_{t}:=\sum_{i=0}^{t-1} Y_i$. 
By Chebychev's inequality and Theorem~\ref{thm:rio_var} we have,
\begin{align}
\Pr\left (\left |\mu_t({\bf X},D)-\mu([D]_{1}^{k})\right| \geq \epsilon \right)&=\Pr(|S_t| \geq t \epsilon) \leq  \frac{\var(S_t)}{t^2\epsilon^2}\leq \frac{4\norm{\boldsymbol{\alpha}}}{t\epsilon^2}
\end{align}
where the last is due to the fact that by definition $\ol{\alpha}(m) \leq \alpha(m),~m \in \N$.
\end{proof}
An empirical estimate of $\alpha_n^{\ell}(m),~m,\ell \in \N,n >m \in \N$  can be obtained as 
\begin{equation}\label{eq:estim}
\widehat{\alpha}_{t,n}^{\ell}({\bf X},m):= \max_{j \in \{1,\ldots,n-m\}}\max_{\substack{A \in \D_{j,\ell}\\B\in \D_{j',\ell} }}\left |\gamma_{t,n}^{m,j}({\bf X},A,B)- \mu_t({\bf X},A) \mu_t({\bf X},B) \right |
\end{equation}
where $j':=n-m-j$, $\mu_t({\bf X},\cdot)$ is given by \eqref{eq:emp} and 
\begin{align}\label{eq:emp_gamma}
\gamma_{t,n}^{m,j}({\bf X},A,B) := \frac{1}{t}\sum_{i=0}^{t-1}\chi_A(X_{in+1},\ldots,X_{i n+j})\chi_B(X_{in+j+m..(i+1)n})
\end{align} 
with $t \geq n$.
When $m,n,\ell \in \N$ are fixed, we only have finitely many cylinder sets to consider in \eqref{eq:estim}, hence the ergodic theorem leads to the following result.
\begin{lemma}\label{lem:ezerg}
Let ${\bf X}$ be a (not necessarily mixing) stationary ergodic process with process distribution $\mu$ and  sequence of $\alpha$-mixing coefficients $\boldsymbol{\alpha}=\langle \alpha(m)\rangle_{m \in \N}$. For every $m,\ell,n \in \N$ it holds that
\begin{equation*}
\lim_{t\rightarrow\infty} \wh{\alpha}_{t,n}^{\ell}({\bf X},m)=\alpha_{n}^{\ell}(m)~,~\mu-\as.
\end{equation*}
\end{lemma}
For each $m,\ell, k\in \N$ define the constant $C_{m,\ell,k}$ by
\begin{equation}\label{eq:constant}
C_{m,\ell,n} :=m (2^{2^{n\ell}+2^{m\ell+1}+1})~.
\end{equation}

When the $\ell_1$ norm of the sequence of $\alpha$-mixing coefficients is finite, we have the following nonasymptotic inequality for
  the empirical version of the approximation of the mixing coefficients.

\begin{proposition}\label{prop:phimb}
Let ${\bf X}$ be a stationary ergodic process with process distribution $\mu$  and sequence of $\alpha$-mixing coefficients $\boldsymbol{\alpha}=\langle \alpha(m)\rangle_{m \in \N}$. For every $m, \ell, n,t\in \N$ and every $\epsilon >0$  we have
\begin{align*}
\Pr(|\widehat{\alpha}_{t,n}^{\ell}({\bf X},m)-\alpha_n^{\ell}(m)| \geq \epsilon) \leq 
\frac{\norm{\boldsymbol{\alpha}}C_{m,\ell,n}}{m t\epsilon^2}~.
\end{align*}
Furthermore, for each $M \in \N$ it holds that,
\begin{align*}
\Pr\left (\left | \sum_{m=1}^M \widehat{\alpha}_{t,n}^{\ell}({\bf X},m)-\alpha_n^{\ell}(m)) \right  | \geq \epsilon \right) \leq  \frac{\norm{\boldsymbol{\alpha}} C_{{{M,\ell,n}}}}{t\epsilon^2}~.
\end{align*}
\end{proposition}
\begin{proof}[Proof of Proposition~\ref{prop:phimb}]
Fix $\epsilon >0$. 
For $t,n \in \N$ define
\begin{equation*}
\Omega_{t,n}:= \left \{\max_{D \in \D_{n,\ell}}|\mu_t({\bf X},D)-\mu([D]_{1}^{n})| \leq\epsilon/2^{2^{m\ell}}\right \}
\end{equation*}
By Lemma~\ref{lem:empf} and a union bound we obtain,
\begin{equation*}
\Pr(\Omega_{t,n}) \geq 1-(2^{2^{n\ell}+2^{m\ell+1}})\frac{\norm{\boldsymbol{\alpha}}}{t\epsilon^2}~.
\end{equation*}
For all $\omega \in \Omega_{t,n}$, and each $A \in \D_{j,\ell},~B \in \D_{j',\ell},~j \in \{1,\ldots,n-m\},~j':=n-m-j+1$ we have
\begin{align*}
\Big |\gamma_{t,n}^{m,j}({\bf X},A,B)- \mu  ([A]_{1}^{j} \cap [B]_{j+m+1}^{n} ) \Big |&\leq\sum_{C \in \D_{m,\ell}} \left  |\mu_{t}({\bf X},A\times C \times B)-\mu([A\times C \times B]_{1}^{n}) \right |\\&\leq \epsilon~,
\end{align*}
and that
\begin{align*}
\Big |\mu_t({\bf X},A)\mu_t({\bf X},B)&-\mu([A]_{1}^{j})\mu([B]_{j+m+1}^{n})\Big| \\
&\leq |\mu_t({\bf X},A)-\mu([A]_{1}^{j})||\mu_t({\bf X},B)-\mu([B]_{j+m+1}^{n})|\\ &~~\quad\qquad \qquad+\mu([A]_{1}^{j})|\mu_t({\bf X},B)-\mu([B]_{j+m+1}^{n})|\\ &~~\quad \qquad\qquad \qquad +\mu([B]_{j+m+1}^{n}) |\mu_t({\bf X},A)-\mu([A]_{1}^{j})|\\
&\leq \epsilon^2/(2^{2^{m\ell}+1})+2\epsilon/(2^{2^{m\ell}})
\leq \epsilon~.
\end{align*}
Recall the convention that for $j \in \{1,\ldots,n-m\}$ we let $j':=j-n-m$.
We obtain 
\begin{align*}
&\Pr(|\widehat{\alpha}_{t,n}^{\ell}({\bf X},m)-\alpha_n^{\ell}(m)| \geq \epsilon)\\
&=\Pr\left (\left | \max_{j \in \{1,\ldots,n-m\}}\max_{\substack{A \in \D_{j,\ell}\\B\in \D_{j',\ell} }}\left |\gamma_{t,n}^{m,j}({\bf X},A,B)- \mu_t({\bf X},A) \mu_t({\bf X},B) \right |-\alpha_{n}^{\ell}(m) \right|\geq \epsilon \right ) \\
&\leq\Pr\left (\exists j,~A \in \D_{j,\ell},B \in \D_{j',\ell}:\left | \gamma_{t,n}^{m,j}({\bf X},A,B) -\mu ([A]_{1}^{j} \cap [B]_{j+m+1}^{n} )\right|\geq \epsilon \right )\\&\qquad +\Pr\left (\exists j,~A \in \D_{j,\ell},B \in \D_{j',\ell}:\left | \mu_t({\bf X},A) \mu_t({\bf X},B) -\mu ( [A]_{1}^{j} ) \mu( [B]_{j+m}^{n})\right|\geq \epsilon \right ) \\
&\leq 2\Pr(\Omega_{t,n}^c)\\
&\leq (2^{2^{n\ell}+2^{m\ell+1}+1})\frac{\norm{\boldsymbol{\alpha}}}{t\epsilon^2}
\end{align*}
Thus, for each $M \in \N$ we obtain,
\begin{align*}
\Pr(|\sum_{m=1}^M \widehat{\alpha}_{t,n}^{\ell}({\bf X},m)-\alpha_n^{\ell}(m)| \geq \epsilon) &\leq \Pr(\sum_{m=1}^M|\widehat{\alpha}_{t,n}^{\ell}({\bf X},m)-\alpha_{n}^{\ell}(m)| \geq \epsilon) \\ 
&\leq \frac{\norm{\boldsymbol{\alpha}} C_{M,\ell,k}}{t\epsilon^2}
\end{align*}
with $C_{M,\ell,k}:=M (2^{2^{n\ell}+2^{M\ell+1}+1})$ as given by \eqref{eq:constant}.
\end{proof}

\subsubsection{Estimation under finite \texorpdfstring{$\norm{\boldsymbol{\alpha}}$}{Lg}}
Equipped with Propositions \ref{prop:phik} and \ref{prop:phimb}, it is now easy to define  estimates of the individual mixing coefficients $\alpha(m)$ that are consistent whenever $\norm{\boldsymbol{\alpha}}<\infty$.

Let $\langle \ell_t \rangle, \langle n_t \rangle$ for $t \in \N$ be non-decreasing unbounded sequences of positive integers.
Let the sequence of positive numbers $\langle \delta_t \rangle_{t \in \N}$ be such that $\sum_{t=1}^\infty \delta_t<\infty$.
Let $\langle \epsilon_t\rangle_{t \in \N}$ be another sequence of positive numbers such that $\lim_{t\to \infty} \epsilon_t = 0$. 
For a fixed $m \in \N$ and each $t \in \N$ let
\begin{equation}\label{eq:def:tau1}
\tau_t := \frac{C_{m,\ell_t,n_t}}{m\epsilon_t^2 \delta_t}
\end{equation}
and define,
\begin{align}
\wh{\alpha}_t({\bf X},m)&:=\wh{\alpha}_{\tau_t,n_t}^{\ell_t}({\bf X},m).\label{eq:alpha_hm}
\end{align}
\begin{theorem}[$\widehat{\alpha}_t$ is strongly consistent]\label{thm:azal_as1} 
For each $m \in \N$ and any stationary $\alpha$-mixing process ${\bf X}$ with process distribution $\mu$  and sequence of $\alpha$-mixing coefficients $\boldsymbol{\alpha}$ such that $\norm{\boldsymbol{\alpha}}<\infty$ we have
\begin{align*}
 \lim_{t\to \infty} \widehat{\alpha}_t({\bf X}, m)=\alpha(m),~\mu-\as.
 \end{align*}
\end{theorem}
\begin{proof}[Proof of Theorem~\ref{thm:azal_as1}]
Fix $\epsilon >0$.  
As follows from Proposition~\ref{prop:phik}, there exist $L_{\epsilon},~N_{\epsilon} \in \N$ 
such that for all $n \geq N_{\epsilon}$ and all $\ell \geq L_{\epsilon}$ we have,
\begin{equation} \label{thm2:eq:ineq1}
 \left |\alpha_n^{\ell}(m)-\alpha(m) \right|\leq \epsilon.
\end{equation}
Let $T_{\epsilon}^{(1)} \in \N$ be such that  $\ell_t \geq L_{\epsilon},~n_t \geq N_{\epsilon}$ for $t \geq T_{\epsilon}^{(1)}$. 
Define the sequence of events 
\begin{equation*}
E_t :=\left \{\left |\widehat{\alpha}_t({\bf X},m)-\alpha_n^{\ell}(m) \right| \geq \epsilon_t \right \},~t \in \N.
\end{equation*} 
Since as specified by \eqref{eq:def:tau1} $\tau_t=\frac{C_{m,\ell_t,n_t}}{\epsilon_t^2 \delta_t}$, then by Proposition~\ref{prop:phimb} for each $t \in \N$ we have 
\begin{equation}\label{eq:prob_Et}
\Pr(E_t)\leq \frac{\norm{\boldsymbol{\alpha}}C_{m,\ell_t,n_t}}{m \tau_t\epsilon_t^2} \leq \frac{\norm{\boldsymbol{\alpha}}\delta_t}{m} .
\end{equation} 
Let $\delta:=\sum_{t \in \N}\delta_t$ and note that since $\langle \delta_t \rangle_{t \in \N}$ is chosen to be summable we have $\delta<\infty$.  
Define 
\[E:=\limsup_{t \to \infty} E_t=\bigcap_{t=1}^{\infty}\bigcup_{t'=t}^{\infty} E_{t'}.\] 
By \eqref{eq:prob_Et} we have $\sum_{t \in \N}\Pr(E_t) \leq \frac{\norm{\boldsymbol{\alpha}}}{m} \sum_{t\in \N} \delta_t = \frac{\norm{\boldsymbol{\alpha}}}{m}\delta <\infty$, since $\norm{\boldsymbol{\alpha}} < \infty$. Therefore, by the Borel-Cantelli Lemma we obtain $\Pr(E)=0$. 
As a result, there exists some $T^*$ (which depends on the sample-path) such that for all $t \geq T^*$ we have 
\begin{equation}\label{thm2:eq:ineq2}
\left |\widehat{\alpha}_t({\bf X},m)-\alpha_{n_t}^{\ell_t}(m) \right| \leq \epsilon_t,~\quad \mu-\as.
\end{equation} 
Let $T_{\epsilon}^{(2)} \in \N$ be large enough such that for all $ t \geq T_{\epsilon}^{(2)}$ we have $\epsilon_t \leq \epsilon/2$ .    
Take $t \geq \max \{T^*,T_{\epsilon}^{(1)}, T_{\epsilon}^{(2)}\}$, with probability it holds that,
\begin{align}
\left |\widehat{\alpha}_t({\bf X},m)-\alpha(m)\right| 
& \leq \left |\widehat{\alpha}_t({\bf X},m)-\alpha_{n_t}^{{\ell_t}}(m) \right|  +  \epsilon \label{thm2:eq:last:1}\\
&\leq \epsilon_t+\epsilon \label{thm2:eq:last:2}\\
&\leq \epsilon \nonumber
\end{align}
where,  \eqref{thm2:eq:last:1} follows from \eqref{thm2:eq:ineq1} and \eqref{thm2:eq:last:2} follows from \eqref{thm2:eq:ineq2}.
Since the choice of $\epsilon$ is arbitrary, the result follows. 
\end{proof}
Next, we introduce an estimator for $\norm{\boldsymbol{\alpha}}$ and prove its consistency under the assumption that the process is stationary $\alpha$-mixing, with a summable $\boldsymbol{\alpha}$. 
For $t \in \N$ recall the parameters $\langle \ell_t \rangle, \langle n_t \rangle, \langle \delta_t \rangle, \langle \epsilon_t\rangle$ defined earlier, and let $\langle M_t \rangle_{t \in \N}$ be an increasing sequence of positive integers. 
For each $t \in \N$ let
\begin{equation}\label{eq:params}
\kappa_t:= \frac{C_{M_t,\ell_t,n_t}}{\epsilon_t^2 \delta_t},
\end{equation}
and define,
\begin{align}
\theta_t({\bf X})&:=\sum_{m=1}^{M_t} \wh{\alpha}_{\kappa_t,n_t}^{\ell_t}({\bf X},m).\label{eq:theta_h}
\end{align}
\begin{theorem}[$\theta_t$ is strongly consistent]\label{thm:azal_as2} 
For any stationary $\alpha$-mixing process ${\bf X}$ with process distribution $\mu$  and sequence of $\alpha$-mixing coefficients $\boldsymbol{\alpha}$ such that $\norm{\boldsymbol{\alpha}}<\infty$, it holds that
\begin{align*}
\lim_{t\to \infty} \theta_t({\bf X})=\norm{\boldsymbol{\alpha}},~\mu-\as.
 \end{align*}
\end{theorem}
\begin{proof}[Proof of Theorem~\ref{thm:azal_as2}]
Fix $\epsilon >0$. Since  $\norm{\boldsymbol{\alpha}}<\infty$, there exists some $M_{\epsilon} \in \N$ such that 
\begin{equation}\label{eq:infi_to_fini}
\sum_{m=M_{\epsilon}+1}^{\infty} \alpha(m) \leq \epsilon/5.
\end{equation}
Let $\epsilon':=\frac{\epsilon}{5 M_{\epsilon}}$. 
As follows from Proposition~\ref{prop:phik}, for each $t \in \N$, there exist $L_{\epsilon'},~N_{\epsilon'} \in \N$ 
such that for all $n \geq N_{\epsilon'}$ and all $\ell \geq L_{\epsilon'}$ we have,
\begin{equation}\label{eq:approx_Meps}
\max_{m \in 1,\ldots,M_{\epsilon}} \left |\alpha_n^{\ell}(m)-\alpha(m) \right|\leq \epsilon'.
\end{equation}
Let $T_{\epsilon}^{(1)} \in \N$ be such that  $M_t \geq M_{\epsilon},~\ell_t \geq L_{\epsilon'(t)},~n_t \geq N_{\epsilon'(t)}$ for $t \geq T_{\epsilon}^{(1)}$. 
Define the sequence of events 
\begin{equation}
E_t :=\left \{\left |\theta_t({\bf X})-  \sum_{m=1}^{M_t}{\alpha}_{n_t}^{\ell_t}(m)\right| \geq \epsilon_t \right \},~t \in \N.
\end{equation} 
Since $\kappa_t=\frac{C_{M_t,\ell_t,n_t}}{\epsilon_t^2 \delta_t}$, then by Proposition~\ref{prop:phimb} for each $t \in \N$ we have 
\[\Pr(E_t)\leq \frac{\norm{\boldsymbol{\alpha}} C_{M_t,\ell_t,k_t}}{t\epsilon_t^2} \leq \norm{\boldsymbol{\alpha}} \delta_t.\]
Let $\delta:=\sum_{t \in \N}\delta_t$ and note that since $\langle \delta_t \rangle_{t \in \N}$ is chosen to be summable we have $\delta<\infty$.  
Define 
$E:=\limsup_{t \to \infty} E_t=\bigcap_{t=1}^{\infty}\bigcup_{t'=t}^{\infty} E_{t'}.$
We have $\sum_{t \in \N}\Pr(E_t) \leq \norm{\boldsymbol{\alpha}} \sum_{t\in \N} \delta_t = \norm{\boldsymbol{\alpha}} \delta <\infty$. Therefore, by the Borel-Cantelli Lemma we obtain $\Pr(E)=0$. 
As a result, there exists some $T^*$ (which depends on the sample-path) such that for all $t \geq T^*$ we have 
\begin{equation}
\left |\theta_t({\bf X})-\sum_{m=1}^{M_t}\alpha_{n_t}^{\ell_t}(m) \right| \leq \epsilon_t,~\quad \mu-\as~. 
\end{equation} 
Let $T_{\epsilon}^{(2)} \in \N$ be large enough so that $\epsilon_t \leq \epsilon/5$ for $ t \geq T_{\epsilon}^{(2)}$.    
With probability one, for all $t \geq \max \{T^*,T_{\epsilon}^{(1)}, T_{\epsilon}^{(2)}\}$ we obtain, 
\begin{align}
\Big |\theta_t({\bf X})-\norm{\boldsymbol{\alpha}} \Big| 
& \leq \Big |\theta_t({\bf X})-\sum_{m=1}^{M_t} \alpha_{n_t}^{\ell_t}(m) \Big|  +   \Big |\sum_{m=1}^{M_t} \alpha_{n_t}^{\ell_t}(m)-\sum_{m=1}^{M_t} \alpha(m) \Big|+\epsilon/5 \nonumber\\
& \leq  \sum_{m=1}^{M_{\epsilon}}  \Big|\alpha_{n_t}^{\ell_t}(m)- \alpha(m) \Big|+\sum_{m=M_{\epsilon}+1}^{M_{t}}  |\alpha_{n_t}^{\ell_t}(m)- \alpha(m) |+2\epsilon/5\nonumber\\
& \leq  \sum_{m=1}^{M_{\epsilon}}   |\alpha_{n_t}^{\ell_t}(m)- \alpha(m) |+\sum_{m=M_{\epsilon}+1}^{M_{t}}  \alpha_{n_t}^{\ell_t}(m)+\sum_{m=M_{\epsilon}+1}^{M_{t}}  \alpha(m) +2\epsilon/5 \label{eq:cool_trick1}\\
&\leq \epsilon. \label{eq:cool_trick2}
\end{align}
where \eqref{eq:cool_trick1} follows from the fact that $\alpha_n^{\ell}(m) \geq 0$ and $\alpha(m) \geq 0$ for all $n,\ell,m \in \N$, and \eqref{eq:cool_trick2} follows from \eqref{eq:approx_Meps} and from observing that by definition $\alpha_n^{\ell}(m) \leq \alpha(m)$ for all $n,\ell,m \in \N$ together with \eqref{eq:infi_to_fini}.
\end{proof}

\subsubsection{Estimation under arbitrary \texorpdfstring{$\norm{\boldsymbol{\alpha}}$}{Lg}}

  Next we address the case when one does not have the guarantee that the $\ell_1$ norm of the sequence of $\alpha$-mixing coefficients
  is finite. This case is significantly more complex, since the conclusion of Proposition \ref{prop:phimb} is potentially
  vacuous, irrespective of the sample size. To circumvent this difficulty, we construct a non-decreasing sequence of
  estimates of $\norm{\boldsymbol{\alpha}}$-- starting with the trivial lower bound $0$ --
  by performing a sequence of tests. The value of the estimate is only increased if there is sufficient evidence
  that the value of $\norm{\boldsymbol{\alpha}}$ is indeed larger than the current estimate. This is
  possible since if $\norm{\boldsymbol{\alpha}}$ happens to be small, then its empirical estimator must be
  small thanks to Proposition \ref{prop:phimb}.

  With the objective to make the presentation more transparent, we start with a simpler estimate that
  is ``weakly consistent'' in the sense that if one prescribes a tolerated probability of error $\delta$, then
  one may construct a corresponding estimator that converges to $\norm{\boldsymbol{\alpha}}$
  with probability at least $1-\delta$. Then we introduce a more complex estimator that is almost surely
  consistent.
  
  In order to introduce the ``weak'' estimator, let $\delta \in (0,1)$ be the allowed probability of error,
  and let $\delta_t$ be positive numbers such that $\sum_{t=1}^{\infty}\delta_t=\delta$. Fix a countable dense subset $\S$ of $[0,\infty)$ and let $\S=\{s_1,s_2,\ldots\}$ be an arbitrary enumeration of $\S$. Given a process ${\bf X}$, set $\ZZ_0({\bf X}):=0$, and for each $t \in \N$ define
\begin{align}\label{eq:hatZt}
&\ZZ_t({\bf X})= \left\{ \begin{array}{ll}
                                     \max\{s_t,\ZZ_{t-1}({\bf X})\} & \text{if} \quad  \theta_{t}({\bf X}) > s_t+ \epsilon_t \sqrt{s_t}\\
                                      \ZZ_{t-1}({\bf X}) & \text{otherwise,}
                                    \end{array}\right.
\end{align}
with $\epsilon_t$ being a decreasing positive sequence converging to zero and $\theta_t$ specified by \eqref{eq:theta_h}.
Clearly, $\langle \ZZ_t({\bf X}) \rangle_{t \in \N}$ form a non-decreasing sequence that is either convergent or diverges to infinity.

\begin{theorem}[$\ZZ_t$ is weakly consistent]
\label{thm:whp}
Let ${\bf X}$ be a (not necessarily mixing) stationary ergodic process with process distribution $\mu$  and sequence of $\alpha$-mixing coefficients $\boldsymbol{\alpha}$. With probability at least $1-\delta$, 
\[
    \lim_{t\to\infty} \ZZ_t({\bf X}) = \norm{\boldsymbol{\alpha}}~.
\]
\end{theorem}

\begin{proof}[Proof of Theorem~\ref{thm:whp}]
First, 
we prove that with probability at least $1-\delta$ we have
\begin{equation}\label{eq:obj}
\lim_{t\to\infty} \ZZ_t({\bf X}) \le \norm{\boldsymbol{\alpha}}.
\end{equation} 
If $\norm{\boldsymbol{\alpha}} =\infty$ then there is nothing to prove, so we may
assume $\norm{\boldsymbol{\alpha}} < \infty$.
Consider any $s_t > \norm{\boldsymbol{\alpha}}$. Then
\begin{align}
\PROB\left(  \theta_{t}({\bf X}) > s_t+ \epsilon_t \sqrt{s_t} \right)     
   & \le  
  \PROB\left(  \theta_{t}({\bf X}) >\sum_{m=1}^{M_t} \alpha(m) + \epsilon_t \sqrt{s_t} \right)\label{eq:gl:partialsum}\\
& \le  
\frac{\norm{\boldsymbol{\alpha}} C_{M_t,\ell_t,n_t}}{\kappa_ts_t \epsilon_t^2} \label{eq:gl:prop2} \\
& \le  \delta_t~, \label{eq:gl:taup} 
\end{align}
where \eqref{eq:gl:partialsum} follows from the fact that $\sum_{m=1}^{M_t}\alpha(m)\leq \norm{\boldsymbol{\alpha}}<s_t$, \eqref{eq:gl:prop2} follows from Proposition~\ref{prop:phimb} and \eqref{eq:gl:taup} from the choice of $\kappa_t$ given by \eqref{eq:params}, together with noting that $s_t>\norm{\boldsymbol{\alpha}}$.
Observing that $\sum_{t \in \N} \delta_t=\delta$, the union bound implies
\[
  \PROB\left (\left\{  \exists s_t > \norm{\boldsymbol{\alpha}}: \theta_{t}({\bf X}) > s_t+ \epsilon_t \sqrt{s_t} \right\}\right)    \le \delta~,  
\]
and \eqref{eq:obj} follows.

It remains to prove that $\lim_{t\to\infty} \ZZ_t({\bf X})\ge \norm{\boldsymbol{\alpha}}$ with probability one. First assume that $\norm{\boldsymbol{\alpha}} < \infty$. 
We show that, with probability one, for any $\eta < \norm{\boldsymbol{\alpha}}$, there exists some $s_t> \eta$ such that 
$\theta_{t}({\bf X}) > s_t+ \epsilon_t \sqrt{s_t}$.
Fix some $\eta< \norm{\boldsymbol{\alpha}}$ and let $\rho=\norm{\boldsymbol{\alpha}}-\eta$.  
Recall that $\langle \epsilon_t\rangle$ is a decreasing sequence of positive real numbers with $\lim_{t\rightarrow\infty}\epsilon_t=0$. Therefore, there exists some $T_1$ such that for all $t\geq T_1$,
\begin{equation}\label{eq:exhaustive:eps}
 \norm{\boldsymbol{\alpha}} - \frac{7 \rho}{8} + \epsilon_t \sqrt{\norm{\boldsymbol{\alpha}} - \frac{7 \rho}{8}} < \norm{\boldsymbol{\alpha}} - \frac{3 \rho}{4}~.
\end{equation}
By Theorem~\ref{thm:azal_as2} there exists some (random) $T_2 \in \N$
such that, for all $t \geq T_2$ with probability one we have,
\begin{equation}\label{eq:exhaustive:const}
   \theta_t({\bf X})\ge \norm{\boldsymbol{\alpha}} - \frac{3 \rho}{4}.
\end{equation}
Moreover, with $\S$ chosen to be a countable dense subset of $[0,\infty)$ we can find some $t\geq \max\{T_1,T_2\}$ such that $s_t \in (\eta, \norm{\boldsymbol{\alpha}} - 7 \rho/8)$. It follows that
\begin{eqnarray*}
    s_t + \epsilon_t \sqrt{s_t}  & < & \norm{\boldsymbol{\alpha}} - \frac{3\rho}{4} \\
& \le & 
\theta_{t} ({\bf X})~,  
\end{eqnarray*}
where the first inequality follows from the fact that $s_t < \norm{\boldsymbol{\alpha}} - 7 \rho/8$ and the second inequality follows from \eqref{eq:exhaustive:const}. As a result 
we obtain $\ZZ_t({\bf X}) > \eta$ as claimed.
Finally, suppose that $\norm{\boldsymbol{\alpha}}=\infty$ and fix some $\eta \in \N$. 
 Since $\norm{\boldsymbol{\alpha}}=\infty$, there exists some $M^{\star}$ such that 
 \begin{equation*}
 \sum_{m=1}^{M^{\star}}\alpha(m) \geq \eta+1~.
 \end{equation*}
 By Proposition~\ref{prop:phik} there exist $n^{\star},~\ell^{\star} \in \N$ such that 
 $
 \max_{m \in 1,\ldots,M^{\star}}|\alpha_{n^{\star}}^{\ell^{\star}}(m)-\alpha(m)|\leq \frac{1}{M^{\star}}
 $ leading to,
 \begin{equation*}
 \sum_{m=1}^{M^{\star}}\alpha_{n^{\star}}^{\ell^{\star}}(m) \geq \eta~.
 \end{equation*}
 Let $T'_1$  be large enough so that $M_t \geq M^{\star}$,~$n_t \geq n^{\star}$, and $\ell_t\geq \ell^{\star}$ for all $t \geq T'_1$. 
 Fix some $\epsilon^*>0$. By Lemma~\ref{lem:ezerg} there exists some (random) $N_{\epsilon^*}$ such that for all $u \geq N_{\epsilon^*}$ with probability one  we have
 $
 \max_{m \in 1,\ldots,M^{\star}}|\wh{\alpha}_{u,n^{\star}}^{\ell^{\star}}(m)-\alpha_{n^{\star}}^{\ell^{\star}}(m)|\leq \epsilon^*/{2M^{\star}} 
 $
 so that 
 \begin{equation}\label{eq:exhaustive:infnorm1}
 \sum_{m=1}^{M^{\star}}\wh{\alpha}_{u,n^{\star}}^{\ell^{\star}}({\bf X},m) \geq \eta-\epsilon^*/2.
 \end{equation}
 Since $\epsilon_t \downarrow 0$, there exists some $T'_2$ such that for all $t\geq T'_2$ it holds that
 \begin{equation}\label{eq:exhaustive:infnorm2}
  \eta - \epsilon^* + \epsilon_t \sqrt{\eta - \epsilon^*} < \eta - \epsilon^*/2.
 \end{equation}
 Let $T'_3$ be large enough so that $\kappa_t$ given by \eqref{eq:params} is larger that $N_{\epsilon^*}$ for all $t\geq T'_3$. 
 Recalling that $\S$ is a countable dense subset of $[0,\infty)$ we can find some $t\geq \max\{T'_1,T'_2,T'_3\}$ such that $s_t \in (\eta-2\epsilon^*, \eta-\epsilon^*)$.
 With probability one it holds that,
 \begin{align}
 s_t+\epsilon_t\sqrt{s_t}
 &<\eta-\epsilon^*/2\label{eq:exhaustive:infnorm3}\\
 &\leq \sum_{m=1}^{M^{\star}}\wh{\alpha}_{\kappa_t,n^{\star}}^{\ell^{\star}}({\bf X},m)\label{eq:exhaustive:infnorm4}\\
 &\leq \sum_{m=1}^{M_t} \wh{\alpha}_{\kappa_t,n_t}^{\ell_t}({\bf X},m)\label{eq:exhaustive:infnorm5}\\
 &= \theta_t({\bf X})\nonumber
 \end{align}
 where \eqref{eq:exhaustive:infnorm3} follows from \eqref{eq:exhaustive:infnorm2},  
 \eqref{eq:exhaustive:infnorm4} follows from \eqref{eq:exhaustive:infnorm1},  
 and \eqref{eq:exhaustive:infnorm5} follows from
 the fact that $t \geq T'_1$ (and hence $M_t \geq M^{\star}$,~$n_t \geq n^{\star}$, and $\ell_t\geq \ell^{\star}$) together with the monotonicity of $\wh{\alpha}_{k,n}^{\ell}(m)$. 
 Since the above holds for any $\eta \in \N$, the result follows.
\end{proof}

Next we present the main result of the paper, an estimator for $\norm{\boldsymbol{\alpha}}$
which, as shown below, is strongly consistent regardless of whether this norm is finite or infinite.
To this end, we build on the construction of $\ZZ_t({\bf X})$, which gives rise to our third estimator of $\norm{\boldsymbol{\alpha}}$. Recall that $\S$ is a countable dense subset of $[0,\infty)$.
Define a sequence $\langle u_t\rangle$ such that $u_t \in \S$ for all $t$ and for every $s\in \S$, the set $\{t\in \N: u_t=s\}$ is infinite. The sequence $\langle u_t\rangle$ ``visits'' each element of $\S$ infinitely often. One such example is:
\begin{align*}
 u_1& =s_1  \\
u_2& =s_1,  u_3= s_2 \\
u_4& =s_1,  u_5= s_2, u_6 = s_3 \\
& \vdots
\end{align*}
The estimator $\Z_t$ given by \eqref{eq:wtz} below keeps track of a ``bit value'' $b_t(s) \in \{0,1\}$ assigned to each $s\in \S$. At each time instance, $b_t(s)$ marks the outcome of the test the last time the value $s$ was tested.
With $b_0(s)=1$ for all $s\in S$, and $\Z_0({\bf X}):=0$, let 
\[
      b_t(s) = \left\{ \begin{array}{ll}
                            1 & \text{if $u_t=s$ and } \quad \theta_t ({\bf X})\leq  s+ \epsilon_t \sqrt{s}\\
                            0 & \text{if $u_t=s$ and } \quad \theta_t({\bf X}) > s+ \epsilon_t \sqrt{s}\\
                            b_{t-1}(s) & \text{otherwise}
                            \end{array} \right.
\]
for each $t \in \N$, and define
\begin{equation}\label{eq:wtz}
      \Z_t({\bf X})= \inf \left\{ s\in \S: b_t(s) = 1 \right\}~,
\end{equation}
where the infimum is defined to be zero if the set is empty.
\begin{theorem}[$\Z_t$ is strongly consistent]\label{thm:exhaustive_strong}
Let ${\bf X}$ be a (not necessarily mixing) stationary ergodic process with process distribution $\mu$ and sequence of $\alpha$-mixing coefficients $\boldsymbol{\alpha}$. 
We have,
\begin{align*}
    \lim_{t\to\infty} \Z_t({\bf X})= \norm{\boldsymbol{\alpha}}, \quad \mu-\as.
\end{align*}
\end{theorem}

\begin{proof}[Proof of Theorem~\ref{thm:exhaustive_strong}]
First we show that 
\begin{equation}
\label{eq:limsup}
    \limsup_{t\to\infty} \Z_t({\bf X}) \le \norm{\boldsymbol{\alpha}}  \quad \text{with probability one}.
\end{equation}
This trivially holds when $\norm{\boldsymbol{\alpha}}=\infty$. So, suppose that $\norm{\boldsymbol{\alpha}}<\infty$. By an argument similar to that given in the proof Theorem~\ref{thm:whp}, whenever $u_t > \norm{\boldsymbol{\alpha}}$, we have
\[
   \PROB\left(  \theta_{t}({\bf X}) > u_t+ \epsilon_t
     \sqrt{u_t} \right) \le   \delta_t~.
\]
Since $\sum_{t\ge 1} \delta_t < \infty$, by the Borel-Cantelli lemma, we have
\[
    \sum_{t: u_t> \norm{\boldsymbol{\alpha}}} \chi\{{\theta_{t}({\bf X}) > u_t+ \epsilon_t
     \sqrt{u_t} }\} < \infty,~\mu-\as.
\] 
Hence, with probability one, there exists a (random) time
$T_1 \in \N$ such that, for all $t>T_1$, 
\begin{equation}\label{eq:test_is_correct}
\theta_{t}({\bf X}) \le  u_t+ \epsilon_t
     \sqrt{u_t}    \quad \text{whenever $u_t > \norm{\boldsymbol{\alpha}}$}~.
\end{equation}
In particular, let
\begin{equation*}
E:=\{s\in \S \cap (\norm{\boldsymbol{\alpha}},\infty):b_{T_1+1}(s)=0\}~,
\end{equation*}
and observe that $|E|<\infty$. Therefore, there exists some time $T_2>T_1$ such that,
for every $s\in E$ there exists $t\in \{T_1+1,\ldots,T_2\}$ with $u_t=s$. 
As follows from \eqref{eq:test_is_correct}, for all $t
\ge T_2$, with probability one we have $b_t(s)=1$ for all $s>\norm{\boldsymbol{\alpha}}$ and (\ref{eq:limsup}) follows.
The fact that 
\begin{equation}\label{eq:liminf}
\liminf_{t\to\infty} \Z_t({\bf X}) \ge \norm{\boldsymbol{\alpha}}~,  \quad \as~,
\end{equation}
for $\norm{\boldsymbol{\alpha}}<\infty$ follows analogously. To see this, take any $\eta<\norm{\boldsymbol{\alpha}}<\infty$ and define $\rho:=\norm{\boldsymbol{\alpha}}-\eta$. 
As in the proof of Theorem~\ref{thm:whp}, there exists some $s \in \S \cap (\eta,\norm{\boldsymbol{\alpha}}-7\rho/8)$ and a (random) time $T'_1$ such that for all $t\geq T'_1$,  with probability one, we have
\begin{equation}\label{eq:test_is_correct2}
\theta_{t}({\bf X}) >  u_t+ \epsilon_t
     \sqrt{u_t}    \quad \text{whenever $u_t \leq s$}~.
\end{equation}
The set $E':=\{s \in \S \cap (\eta,\norm{\boldsymbol{\alpha}}-7\rho/8):b_{T'_1+1}(s)=1\}$ can only have finitely many elements, each of which will eventually be visited; that is, we can find some $T'_2>T'_1$ such that for every $s \in E'$ there exists some $t \in \{T'_1+1,\ldots,T'_2\}$ with $u_t=s$. By \eqref{eq:test_is_correct2} for all $t\geq T'_2$ with probability one it holds that $b_t(s)=0$ for all~$s \in \S \cap (\eta,\norm{\boldsymbol{\alpha}}-7\rho/8)$ and \eqref{eq:liminf} follows. 
Now suppose that $\norm{\boldsymbol{\alpha}}=\infty$.  
By an argument analogous to the above,  for each $\eta \in \N$ and every $\epsilon^* >0$ there exists some $s \in (\eta-2\epsilon^*,\eta-\epsilon^*)$ and a (random) $T''_1 \in \N$ such that for all $t \geq T''_1$ 
with probability one we have 
\begin{equation}\label{eq:test_is_correct3}
\theta_{t}({\bf X}) >  \eta-\epsilon^* \geq  u_t+ \epsilon_t
     \sqrt{u_t}    \quad \text{whenever $u_t \leq s$}~.
\end{equation}
Again, by \eqref{eq:test_is_correct3}  and noting that the set $E'':=\{s \in \S \cap (\eta-2\epsilon^*,\eta-\epsilon^*):b_{T''_1+1}(s)=1\}$ has finite cardinality, we can find some $T''_2$ such that for all $t \geq T''_2$ with probability one we have $b_t(s)=0$ for all $s \in \S \cap (\eta-2\epsilon^*,\eta-\epsilon^*)$. Since this holds for every $\eta \in \N$ and any $\epsilon^*>0$, it follows that when $\norm{\boldsymbol{\alpha}}=\infty$ we have $\liminf_{t\rightarrow \infty}\Z_t({\bf X})\geq \infty,~\mu-\as$. This completes the proof.
\end{proof}
\subsection{Estimating \texorpdfstring{$\beta(m)$}{Lg} and \texorpdfstring{$\norm{\boldsymbol{\beta}}$}{Lg}}\label{sec:est_b}
In this section, we derive estimators concerning the $\beta$-mixing coefficients of a stationary ergodic process. 
For the most part, the results are analogous to those given in Section~\ref{sec:est_a}. An overview of the similarities and differences between the two settings is provided in the introduction.

We start by defining an approximation of the $\beta$-mixing coefficients of ${\bf X}$.
Recall the notation $\D_{k,\ell},~k,\ell \in \N$ introduced in the beginning of Section~\ref{sec:est} corresponding to the power-set $\P(\Delta_{k,\ell})$ of the set of dyadic cubes in $\X^k$ of side-length $2^{-\ell}$. 
For $m,\ell \in \N,~n > m \in \N$ and $j \in \{1,\ldots,n-m\}$, let
\begin{equation}\label{eq:approx_b}
\beta_{n,j}^{\ell}(m):= \frac{1}{2} \sum_{A \in \D_{j,\ell}} \sum_{B \in \D_{j',\ell}}\left |\mu \left ([A]_{1}^{j} \cap [B]_{j+m+1}^{n} \right ) -\mu \left( [A]_{1}^{j} \right ) \mu \left( [B]_{j+m+1}^{n}\right )  \right |
\end{equation}
where $j':=n-m-j+1$. The intuition behind this choice follows from the discussion provided in Section~\ref{subsec:estimation}. Recall the $\sigma$-subalgebras of $\F$, denoted $\F_{1}^{j}(\ell)$ and $\F_{j+m}^{n}(\ell)$, with $m, \ell \in \N,~n >m \in \N$ and $j \in \{1,\ldots,n-m\}$, as  given by \eqref{eq:defn_Fj} and \eqref{eq:defn_Fj+m} respectively.   In a manner analogous to the estimation of the $\alpha$-mixing coefficients,  we show via Lemma~\ref{lem:dynk_b} below that to approximate the $\beta$-dependence between between $\F_{1}^{j}(\ell)$ and $\F_{j+m}^{n}(\ell)$, it suffices to use $\beta_{n,j}^{\ell}(m)$ given by \eqref{eq:approx_b}.  Note that the $\pi-\lambda$ argument used in the proof of  Lemma~\ref{lem:dynk} does not carry over to this case where summations are involved in \eqref{eq:approx_b} in place of the suprema in \eqref{eq:approx}. 
\begin{lemma}\label{lem:dynk_b}
For $m, \ell \in \N,~n >m \in \N$  and $j \in \{1,\ldots,n-m\}$ we have 
\[
  \beta_{n,j}^{\ell}(m)=\beta( \F_{1}^{j}(\ell), \F_{j+m}^{n}(\ell))~.
\] 
\end{lemma}
\begin{proof}[Proof of Lemma~\ref{lem:dynk_b}]
Fix $m, \ell \in \N,~n >m \in \N$ and some $j \in \{1,\ldots,n-m\}$. 
Define
$\C_{1}^{j}(\ell):=\left \{[A]_{1}^{j}: A \in \D_{j,\ell} \right \}$ and 
$
\C_{j+m}^{n}(\ell):=\left \{ [B]_{j+m+1}^{n}: B\in \D_{j',\ell},~j':=n-m-j\right\}.$
Observe that $\sigma(\C_{1}^{j}(\ell))=\F_{1}^j(\ell)$ and $\sigma(\C_{j+m}^{n}(\ell))=\F_{j+m}^n(\ell)$, and that $\F_{1}^j(\ell)$ and $\F_{j+m}^n(\ell)$ are atomic, with the elements of the countable collections $\C_{1}^{j}(\ell)$ and $\C_{j+m}^{n}(\ell)$ as their atoms respectively. By Bradley \cite[Vol. 1, Proposition 3.21]{BRA07} we have 
\begin{equation*}
\beta(\F_{1}^j(\ell),\F_{j+m}^n(\ell))=\frac{1}{2}\sum_{U \in \C_{1}^{j}(\ell)}\sum_{V \in \C_{j+m}^{n}(\ell)}|\mu(U \cap V)-\mu(U)\mu(V)|=\beta_{n,j}^{\ell}(m)~.
\end{equation*}
\end{proof}

For $m, \ell \in \N,~n >m \in \N$, let 
\begin{equation}
\beta_n^{\ell}(m)=\max_{j \in \{1,\ldots,n-m\}}\beta_{n,j}^{\ell}(m)
\end{equation}
define an approximation of $\beta(m)$. 
\begin{proposition}\label{prop:phik_b}
For every $m\in \N$ we have
$
\displaystyle \lim_{ n,\ell \rightarrow \infty}\beta^{\ell}_n(m)=\beta(m)~.
$
\end{proposition}
\begin{proof}[Proof sketch of Proposition~\ref{prop:phik_b}]
Fix  $m,\ell \in \N$ and $n >m$.
Observe that for each $j \in \{1,\ldots,n-m\}$ we have $\F_{1}^{j}(\ell) \subseteq \F_{1}^{j}(\ell+1),~\ell \in \N$ and $\F_{j+m}^n(\ell)\subseteq \F_{j+m}^{n+1}(\ell+1),~n,\ell \in \N$ are each a sequence of $\sigma$-algebras with $\bigvee_{\ell}^{\infty}\F_{1}^{j}(\ell) =\sigma(X_{1},\ldots,X_j)$ and 
$\bigvee_{n,\ell}^{\infty}\F_{j+m}^{n}(\ell)=\sigma(\{X_{t}: t \geq j+m+1\})$. 
Therefore, by Bradley \cite[Vol. 1, Proposition 3.18]{BRA07} we obtain, 
\begin{equation}\label{eq:prop318_b}
\lim_{n,\ell \to \infty}\beta(\F_{1}^{j}(\ell),\F_{j+m}^n(\ell))=\beta(\sigma(\{X_{t}: 1\leq t \leq j\}),\sigma(\{X_{t}: t \geq j+m+1\}))~.
\end{equation}
As in Lemma~\ref{prop:phik}, it is straightforward to  check that for each $m \in \N$ we have
\begin{align}\label{eq:supmax_b}
\sup_{n,\ell}\max_{j \in \{1,\ldots,n-m\}}\beta_{n,j}^{\ell}(m)=\sup_{j \in \N}\sup_{n\geq j+m+1} \sup_{\ell}\beta_{n,j}^{\ell}(m)~.
\end{align}
Thus, analogously to the proof of Lemma~\ref{prop:phik}, by \eqref{eq:supmax_b} and Lemma~\ref{lem:dynk_b} we obtain
\begin{align*}
\lim_{n,\ell \rightarrow\infty} \beta_{n}^{\ell}(m) 
&=\beta(m)~.\nonumber
\end{align*}
\end{proof}
An empirical estimate of $\beta_n^{\ell}(m),~m,\ell \in \N,n >m \in \N$  can be obtained as 
\begin{equation}\label{eq:estim_b}
\widehat{\beta}_{t,n}^{\ell}({\bf X},m):= \max_{j \in \{1,\ldots,n-m\}}\frac{1}{2}\sum_{A \in \D_{j,\ell}}\sum_{B\in \D_{j',\ell}}\left |\gamma_{t,n}^{m,j}({\bf X},A,B)- \mu_t({\bf X},A) \mu_t({\bf X},B) \right |
\end{equation}
where $j':=n-m-j$, $\mu_t({\bf X},\cdot)$ and $\gamma_{t,n}^{m,j}({\bf X},\cdot,\cdot)$ are given by \eqref{eq:emp} and \eqref{eq:emp_gamma} respectively.
As with Lemma~\ref{lem:ezerg}, since for fixed $m,n,\ell \in \N$ we only have finitely many cylinder sets to consider in \eqref{eq:estim_b}, a simple application of the ergodic theorem gives the following lemma.
\begin{lemma}\label{lem:ezerg_b}
Let ${\bf X}$ be a (not necessarily mixing) stationary ergodic process with process distribution $\mu$ and  sequence of $\beta$-mixing coefficients $\boldsymbol{\beta}=\langle \beta(m)\rangle_{m \in \N}$. For every $m,\ell,n \in \N$ it holds that
\begin{equation*}
\lim_{t\rightarrow\infty} \wh{\beta}_{t,n}^{\ell}({\bf X},m)=\beta_{n}^{\ell}(m),~\mu-\as.
\end{equation*}
\end{lemma}
For each $m,\ell, k\in \N$ let the constant $\bC_{m,\ell,k}$ be given by
\begin{equation}\label{eq:constant_b}
\bC_{m,\ell,k} :=m (2^{2^{2k\ell+1}+2^{m\ell+1}+2})~.
\end{equation}
An argument analogous to that given for Proposition~\ref{prop:phimb} leads to the following result.
\begin{proposition}\label{prop:phimb_b}
Let ${\bf X}$ be a stationary ergodic process with process distribution $\mu$  and sequence of $\beta$-mixing coefficients $\boldsymbol{\beta}=\langle \beta(m)\rangle_{m \in \N}$. For every $m, \ell, n,t\in \N$ and every $\epsilon >0$  we have
\begin{align*}
\Pr(|\widehat{\beta}_{t,n}^{\ell}({\bf X},m)-\beta_n^{\ell}(m)| \geq \epsilon) \leq \frac{\norm{\boldsymbol{\beta}}\bC_{m,\ell,n}}{m t\epsilon^2}
\end{align*}
with $\bC_{m,\ell,n}$ given by \eqref{eq:constant_b}.
Furthermore, for each $M \in \N$,
\begin{align*}
\Pr\left (\left | \sum_{m=1}^M \widehat{\beta}_{t,n}^{\ell}({\bf X},m)-\beta_n^{\ell}(m)) \right  | \geq \epsilon \right) \leq  \frac{\norm{\boldsymbol{\beta}} \bC_{{{M,\ell,n}}}}{t\epsilon^2}~.
\end{align*}
\end{proposition}
\begin{proof}[Proof sketch of Proposition~\ref{prop:phimb_b}]
The result follows in much the same way as in Proposition~\ref{prop:phimb}, by observing that since $\beta(m)\geq \alpha(m),~m \in \N$, as follows from Lemma~\ref{lem:empf} for all $ D \in \D_{n,\ell},~n,\ell, t \in \N$ and every $\epsilon >0$ we have,
\begin{equation*}
\Pr\left ( \left |\mu_t({\bf X},D)-\mu([D]_{1}^{n}) \right| \geq \epsilon \right)\leq  \frac{4\norm{\boldsymbol{\beta}}}{t\epsilon^2}~.
\end{equation*}
Thus, for a fixed $\epsilon>0$ and $t,n \in \N$ we can deduce that $\Pr(\Omega_{t,n}) \geq 1-(2^{2^{2n\ell+1}+2^{m\ell+1}+2})\frac{\norm{\boldsymbol{\beta}}}{t\epsilon^2}.$
where
$
\Omega_{t,n}:= \left \{\max_{D \in \D_{n,\ell}}|\mu_t({\bf X},D)-\mu([D]_{1}^{n})| \leq \epsilon/(2^{2^{m\ell}+2^{2n\ell}})\right \}.
$ Next, as in the proof of Proposition~\ref{prop:phimb} we can obtain 
\begin{align*}
\Pr(|\widehat{\beta}_{t,n}^{\ell}({\bf X},m)-\beta_n^{\ell}(m)| \geq \epsilon)
&\leq 2\Pr(\Omega_{t,n}^c)\leq (2^{2^{2n\ell+1}+2^{m\ell+1}+2})\frac{\norm{\boldsymbol{\beta}}}{t\epsilon^2}~,
\end{align*}
which together with a union bound gives the second concentration bound as well. 
\end{proof}
As in the previous section, let $\langle \ell_t \rangle, \langle n_t \rangle$ for $t \in \N$ be non-decreasing unbounded sequences of positive integers.
Let the sequence of positive real numbers $\langle \delta_t \rangle_{t \in \N}$ be such that $\sum_{t=1}^\infty \delta_t<\infty$.
Let $\langle \epsilon_t\rangle_{t \in \N}$ be another sequence of positive real numbers such that $\lim_{t\to \infty} \epsilon_t = 0$. Furthermore let $\langle M_t \rangle_{t \in \N}$ be an increasing sequence of positive integers. 
For each $t \in \N$ let
\begin{align}
&\btau_t := \frac{\bC_{m,\ell_t,n_t}}{m\epsilon_t^2 \delta_t},~m \in \N
&\text{and}
&{~}
&\bkappa_t:= \frac{\bC_{M_t,\ell_t,n_t}}{\epsilon_t^2 \delta_t},\label{eq:params_b}
\end{align}
with the constant $\bC_{m,\ell,n}$ given by \eqref{eq:constant_b}. 
Define  
\begin{align}
\wh{\beta}_t({\bf X},m)&:=\wh{\beta}_{\btau_t,n_t}^{\ell_t}({\bf X},m)
&\text{and}
&{~}
&\btheta_t({\bf X})&:=\sum_{m=1}^{M_t} \wh{\beta}_{\bkappa_t,n_t}^{\ell_t}({\bf X},m).\label{eq:theta_h_b}
\end{align}
Using Proposition~\ref{prop:phik_b}, Proposition~\ref{prop:phimb_b} and based on the choice of parameters specified by \eqref{eq:params_b}, arguments analogous to those given for Theorem~\ref{thm:azal_as1} and Theorem~\ref{thm:azal_as2}, yield the following result.
\begin{theorem}[$\widehat{\beta}_t$ and $\btheta_t$ are strongly consistent]\label{thm:azal_as1_b} 
For each $m \in \N$ and any stationary $\beta$-mixing process ${\bf X}$ with process distribution $\mu$  and sequence of $\beta$-mixing coefficients $\boldsymbol{\beta}$ such that $\norm{\boldsymbol{\beta}}<\infty$ we have
\begin{align*}
 \lim_{t\to \infty} \widehat{\beta}_t({\bf X}, m)=\beta(m),~\mu-\as,
 \end{align*}
 and 
 \begin{align*}
\lim_{t\to \infty} \btheta_t({\bf X})=\norm{\boldsymbol{\beta}},~\mu-\as.
 \end{align*}
\end{theorem}
Next, we introduce two estimators for $\norm{\boldsymbol{\beta}}$, which are analogous to those given by \eqref{eq:hatZt} and \eqref{eq:wtz} for the estimation of $\norm{\boldsymbol{\alpha}}$. As in the previous section, let $\S:=\{s_1,s_2,\ldots\}$ be a countable dense subset of $[0,\infty)$. Given a process ${\bf X}$, set $\bZZ_0({\bf X}):=0$, and for each $t \in \N$ define
\begin{align}\label{eq:hatZt_b}
&\bZZ_t({\bf X})= \left\{ \begin{array}{ll}
                                     \max\{s_t,\bZZ_{t-1}({\bf X})\} & \text{if} \quad  \btheta_{t}({\bf X}) > s_t+ \epsilon_t \sqrt{s_t}\\
                                      \bZZ_{t-1}({\bf X}) & \text{otherwise,}
                                    \end{array}\right.
\end{align}
with $\btheta_t$ specified by \eqref{eq:params_b}. 
Define a sequence $\langle u_t\rangle$ such that $u_t \in \S$ for all $t$ and for every $s\in \S$, the set $\{t\in \N: u_t=s\}$ is infinite. The sequence $\langle u_t\rangle$ ``visits'' each element of $\S$ infinitely often. 
As in the case of $\Z_t$ given by \eqref{eq:wtz}, the estimator $\bZ_t$ given by \eqref{eq:wtz_b} below keeps track of the the outcome of the test when the value $s\in \S$ was last tested.
Set $\wt{b}_0(s)=1$ for all $s\in S$, and $\bZ_0({\bf X}):=0$, let 
\[
      \wt{b}_t(s) = \left\{ \begin{array}{ll}
                            1 & \text{if $u_t=s$ and } \quad \btheta_t ({\bf X})\leq  s+ \epsilon_t \sqrt{s}\\
                            0 & \text{if $u_t=s$ and } \quad \btheta_t({\bf X}) > s+ \epsilon_t \sqrt{s}\\
                            b_{t-1}(s) & \text{otherwise.}
                            \end{array} \right.
\]
for each $t \in \N$, and define
\begin{equation}\label{eq:wtz_b}
      \bZ_t({\bf X})= \inf \left\{ s\in \S: \wt{b}_t(s) = 1 \right\}
\end{equation}
where the $\inf$ is defined to be zero if the set is empty.
Thanks to Theorem~\ref{thm:azal_as1_b},
Proposition~\ref{prop:phimb_b},
Proposition~\ref{prop:phik_b}, and
Lemma~\ref{lem:ezerg_b}, and based on 
arguments analogous to those given for Theorem~\ref{thm:whp} and Theorem~\ref{thm:exhaustive_strong}, it can e shown that $\bZZ_t$ and $\bZ_t$ are consistent, regardless of the summability of $\boldsymbol{\beta}$. More specifically, we have the following results. 
\begin{theorem}[$\bZZ_t$ and $\bZ_t$ are weakly and strongly consistent respectively]
\label{thm:whp_b}
Let ${\bf X}$ be a (not necessarily mixing) stationary ergodic process with process distribution $\mu$  and sequence of $\beta$-mixing coefficients $\boldsymbol{\beta}$. With $\mu$-probability at least $1-\delta$, we have
\[
    \lim_{t\to\infty} \bZZ_t({\bf X}) = \norm{\boldsymbol{\beta}}.
\]
Moreover, 
\begin{align*}
    \lim_{t\to\infty} \bZ_t({\bf X})= \norm{\boldsymbol{\alpha}}, \quad \mu-\as.
\end{align*}
\end{theorem}
\section{Goodness-of-Fit Testing}\label{sec:gf}
In this section we use the estimators provided in Section~\ref{sec:est} to construct consistent goodness-of-fit tests. 
Consider a stationary ergodic process $\bf X$ with process distribution $\mu$ and sequences of $\alpha$-mixing and $\beta$-mixing coefficients $\boldsymbol{\alpha}=\langle \alpha(m) \rangle_{m \in \N}$ and $\boldsymbol{\beta}=\langle \beta(m) \rangle_{m \in \N}$  respectively.
A mapping $\gamma: \N \rightarrow [0,1]$ is called an $\alpha$-mixing  {\em rate function} for $\mu$ if for each $m \in \N$ we have 
$
\alpha(m) \leq \gamma(m).
$
Similarly, it is said to be a $\beta$-mixing rate function for $\mu$ if for each fixed $m \in \N$ it holds that
$
\beta(m) \leq \gamma(m).
$

We wish to test the null hypothesis $H_0$ that $\gamma$ is an $\alpha$-mixing  (respectively $\beta$-mixing) rate function for $\mu$ against the alternative hypothesis $H_1$ that there exists some $m \in \N$ such that $\alpha(m)>\gamma(m)$ (respectively $\beta(m)>\gamma(m)$) .  
More formally, let $\fC$ be the class of stationary $\alpha$-mixing processes whose sequence of $\alpha$-mixing coefficients $\boldsymbol{\alpha}=\langle \alpha(m) \rangle_{m \in \N}$ is summable, that is, such that $\norm{\boldsymbol{\alpha}}<\infty$. Similarly, let $\bfC$ be the class of stationary $\beta$-mixing processes whose sequence of $\beta$-mixing coefficients is such that $\norm{\boldsymbol{\beta}}<\infty$. 
For a rate function $\gamma$, let $\fR_{\gamma}$ (respectively $\bfR_{\gamma}$) be the class of all processes which have $\gamma$ as an $\alpha$-mixing (respectively $\beta$-mixing) rate function. 
We construct  a sequence  $\langle g_t \rangle_{t \in \N}$ of functions $g_t: \X^\N \to \{+1,-1\},~t\in \N$ each measurable with respect to the filtration $\F_t$ such that given a sample $\bf{X}$ generated by $\mu \in \fC$,
with probability $1$, produce
\begin{align}\label{eq:contr_const}
&\lim_{t\to \infty}g_t(\bf X)=
\begin{cases} +1 & \text{if} \ \mu \in \fC \cap \fR_{\gamma}\\ -1 & \text{if} \  \mu \in \fC \setminus \fR_{\gamma}~. \end{cases} 
\end{align}
We call such a sequence of function satisfying \eqref{eq:contr_const} a strongly consistent goodness-of-fit test for rate function $\gamma$ on $\fC$. 
Analogously, we construct a sequence $\langle \wt{g}_t \rangle_{t \in \N}$ of $\F_t$-measurable functions $\wt{g}_t: \X^\N \to \{+1,-1\},~t\in \N$ such that with  probability $1$,
\begin{align}\label{eq:contr_const_b}
&\lim_{t\to \infty}\wt{g}_t(\bf X)=
\begin{cases} +1 & \text{if} \ \mu \in \bfC \cap \bfR_{\gamma}\\ -1 & \text{if} \ \mu \in \bfC \setminus \bfR_{\gamma}~. \end{cases} 
\end{align}
To construct our tests, let us recall some relevant notation from the previous sections.  For $t \in \N$ let $\langle M_t \rangle, \langle \ell_t \rangle, \langle n_t \rangle$ be increasing sequences of positive integers. Take $\delta \in (0,1)$ and let the sequence of positive real numbers $\langle \delta_t \rangle_{t \in \N}$ be such that $\sum_{t=1}^\infty \delta_t=\delta$.
Let $\langle \epsilon_t\rangle_{t \in \N}$ be another sequence of positive numbers such that $\lim_{t\to \infty} \epsilon_t = 0$. 
For each $m \in \N$ and $t \in \N$ recall the estimators 
\begin{align}\label{eq:alpha1hat}
&\wh{\alpha}_t({\bf X},m):=\wh{\alpha}_{\tau_t,n_t}^{\ell_t}({\bf X},m)
&~
&\wh{\beta}_t({\bf X},m):=\wh{\beta}_{\wt{\tau}_t,n_t}^{\ell_t}({\bf X},m)
\end{align}
respectively specified by \eqref{eq:alpha_hm} and \eqref{eq:theta_h_b}, 
with $\tau_t$ given by \eqref{eq:def:tau1} and $\wt{\tau}_t$ by \eqref{eq:params_b}.
For a given process $\bf{X}$ and every $t \in \N$ define 
\begin{equation}\label{eq:gamma_rate_a}
g_t(\bf{X}):=\begin{cases} 
+1& \text{if} \ \wh{\alpha}_t ({\bf X},m)\leq \gamma(m)+\epsilon_t,~m \in \{1,\ldots,M_t\}\\
-1& \text{otherwise}
\end{cases}
\end{equation}
and
\begin{equation}\label{eq:gamma_rate_b}
\wt{g}_t(\bf{X}):=\begin{cases} 
+1& \text{if} \ \wh{\beta}_t ({\bf X},m)\leq \gamma(m)+\epsilon_t,~m \in \{1,\ldots,M_t\}\\
-1&  \text{otherwise}
\end{cases}
\end{equation}
\begin{theorem}\label{thm:gamma_rate}
The sequence of functions ${g}_t,~t \in \N$ given by \eqref{eq:gamma_rate_a} give rise to a goodness-of-fit test on $\fC$ that is strongly consistent in the sense of \eqref{eq:contr_const}. 
\end{theorem}
\begin{proof}[Proof of Theorem~\ref{thm:gamma_rate}]
Consider a rate function $\gamma: \N \rightarrow [0,1]$ and denote by $\fR_{\gamma}$ the class of processes with $\gamma$ as rate function of their $\alpha$-mixing coefficients. Consider a process measure $\mu \in \fC$ with corresponding sequence of random variables $\bf X$.
First, observe that $g_t$ is $\F_t$-measurable.
To prove that $g_t$ is consistent in the sense of \eqref{eq:contr_const}, we proceed as follows. Consider a process ${\bf X}$ with process distribution $\mu$. 
Let $$E_t:=\{\exists m \in \{1,\ldots,M_t\} \st \wh{\alpha}_t({\bf X},m)> \gamma(m)+ \epsilon_t\}~.$$ 
If $\mu \in \fC \cap \fR_{\gamma}$, we have 
\begin{align}
\Pr(E_t)
& \leq \sum_{m=1}^{M_t}\Pr(\wh{\alpha}_t({\bf X},m) > \gamma(m)+\epsilon_t)\nonumber \\
& \leq \sum_{m=1}^{M_t}\Pr(\wh{\alpha}_t({\bf X},m) > \alpha(m)+\epsilon_t)\label{eq:muisinRgamma}\\
& \leq \sum_{m=1}^{M_t}\Pr(|\wh{\alpha}_t({\bf X},m) - \alpha_{n_t}^{\ell_t}(m)|>\epsilon_t)\label{eq:alpha_n_alpha} \\
&\leq \norm{\boldsymbol{\alpha}}\delta_t~, \label{eq:upper-bound1}
\end{align}
where \eqref{eq:muisinRgamma} follows from the fact that $\mu \in \fR_{\gamma}$, and \eqref{eq:alpha_n_alpha}  from observing that $\alpha_{n_t}^{\ell_t}(m) \leq \alpha(m),~m\in \N$, and \eqref{eq:upper-bound1} follows from Proposition~\ref{prop:phimb}.
Noting that $\norm{\boldsymbol{\alpha}}<\infty$ and $\delta_t$ is summable, we have $\sum_{t \in \N}\Pr(E_t) \leq \norm{\boldsymbol{\alpha}} \sum_{t\in \N} \delta_t = \norm{\boldsymbol{\alpha}} \delta <\infty$. Therefore, by the Borel-Cantelli Lemma we obtain $\Pr(\limsup_{t\rightarrow \infty} E_t)=0$. 
This means that there exists some $\tau$ (which depends on the sample-path) such that for all $t \geq \tau$ we have 
\begin{equation}
g_t({\bf X})=+1, ~\mu-\as~.
\end{equation} 
On the other hand, suppose $\mu \in \fC \setminus \fR_{\gamma}$ and observe that there exists some $m^* \in \N$ and some $\delta^* \in (0,1]$ such that $\alpha(m^*)-\gamma(m^*)=\delta^*$. 
Let $$\wt{E}_t:=\{\forall m \in \{1,\ldots,M_t\} \st \wh{\alpha}_t(m) \leq \gamma(m)+ \epsilon_t\}~.$$ 
Recalling that $\langle \ell_t \rangle$ and $\langle n_t \rangle$ are increasing sequences, by Proposition~\ref{prop:phik} there exists some $T_1$ such that for all $t \geq T_1$ we have 
\begin{equation}\label{eq:approx_to_mstar}
|\alpha_{n_t}^{\ell_t}(m^*)-\alpha(m^*)| \leq \delta^*/2~.
\end{equation}
Since $\langle M_t \rangle$ is an increasing sequence, we can find some $T_2 \in \N$ such that $m^*\leq M_t$ for all $t \geq T_2$. Moreover, since $\langle \epsilon_t \rangle$ decreases with $t$ there exists some $T_3$ such that $\epsilon_t \leq \delta^*/4$ for all $t \geq T_3$.
Let $T^*:=\max\{T_1,T_2,T_3\}$. For $t \geq T^*$ we have,
\begin{align}
\Pr(\wt{E}_t) &\leq \Pr(\wh{\alpha}_t({\bf X},m^*)\leq \gamma(m^*)+\epsilon_t)\nonumber \\
&\leq \Pr(\wh{\alpha}_t({\bf X},m^*)\leq \gamma(m^*)+\delta^*/2-\epsilon_t) \label{eq:alternative1}\\
& \leq \Pr(\wh{\alpha}_t({\bf X},m^*)-\alpha_{n_t}^{\ell_t}(m^*)\leq -\epsilon_t) \label{eq:alternative2}\\
&\leq \Pr(|\wh{\alpha}_t({\bf X},m^*)-\alpha_{n_t}^{\ell_t}(m^*)|\geq \epsilon_t) \nonumber\\
&\leq \norm{\boldsymbol{\alpha}}\delta_t\label{eq:alternative4}
\end{align}
where \eqref{eq:alternative1} follows from the fact that $\epsilon_t \leq \delta^*/4$ for $t\geq T_3$ so that $\delta^*/2-\epsilon_t \geq \epsilon_t$, \eqref{eq:alternative2}  follows from  \eqref{eq:approx_to_mstar} 
and \eqref{eq:alternative4} follows from Proposition~\ref{prop:phimb}  and the choice of $\tau_t$. 
In much the same way as above, we have $\sum_{t =T^*}^\infty\Pr(\wt{E}_t) \leq  \norm{\boldsymbol{\alpha}}\delta<\infty$ and by the Borel-Cantelli Lemma we obtain $\Pr(\limsup_{t\rightarrow\infty}\wt{E}_t) \leq \Pr(\cap_{t=T^*}^{\infty}\cup_{t'=t}^\infty \wt{E}_{t'})=0$. Thus, there is some 
(random) $T$ such that for all $t \geq T$ we have 
\begin{equation*}
g_t({\bf X})=-1, \quad \mu-\as.
\end{equation*} 
This proves the statement.
\end{proof}

It is straightforward to check that a similar argument based on Proposition~\ref{prop:phik_b} and Proposition~\ref{prop:phimb_b} leads to the following analogous result concerning $\beta$-mixing rate functions.

\begin{theorem}\label{thm:gamma_rate_b}
  The sequence of functions $\wt{g}_t,~t \in \N$ defined by \eqref{eq:gamma_rate_b} gives rise to a goodness-of-fit test on $\bfC$ that is strongly consistent in the sense of \eqref{eq:contr_const_b}.
\end{theorem}

Next, we propose an asymptotically consistent test for an upper-bound on $\norm{\boldsymbol{\alpha}}$ (respectively $\norm{\boldsymbol{\beta}}$) of a stationary ergodic process that is not necessarily mixing. Denote by $\fE$ the set of all $\X$-valued stationary ergodic processes and observe that $\bfC \subset \fC \subset \fE$. Given a sample ${\bf X}$ generated by $\mu \in \fE$, we wish to test the null hypothesis $H_0$ that $\norm{\boldsymbol{\alpha}}\leq \gamma$ (respectively $\norm{\boldsymbol{\beta}}\leq \gamma$) for some fixed threshold $\gamma \in [0,\infty)$ against the alternative hypothesis that $\norm{\boldsymbol{\alpha}}> \gamma$ ($\norm{\boldsymbol{\beta}}\leq \gamma$).  Let $\langle \zeta_t\rangle_{t \in \N}$ be a decreasing sequence of positive real numbers such that $\lim_{t \rightarrow \infty}\zeta_t=0$.  For a sample ${\bf X}$, define the sequences $\langle f_t\rangle_{t \in \N}$ and $\langle \wt{f}_t\rangle_{t \in \N}$ of $\F_t$-measurable functions
\begin{equation}\label{eq:general_test_se_gamma}
  f_t(\bf{X}):=\begin{cases}
    +1& \text{if} \ \Z_t ({\bf X})\leq \gamma+\zeta_t\\
    -1& \text{otherwise} \end{cases} \end{equation} and \begin{equation}\label{eq:general_test_se_gamma_b} \wt{f}_t(\bf{X}):=\begin{cases}
    +1& \text{if} \ \bZ_t ({\bf X})\leq \gamma+\zeta_t\\
    -1&\text{otherwise}~, \end{cases}
\end{equation}
where $\Z_t ({\bf X})$ is given by \eqref{eq:wtz} and $\bZ_t ({\bf X})$ is given by \eqref{eq:wtz_b}. For each $\gamma \in [0,\infty)$ denote by $\fC_{\gamma}$ the subclass of $\fC$ corresponding to stationary $\alpha$-mixing processes whose sequences of $\alpha$-mixing coefficients sum to at most $\gamma$. 
A simple argument which relies on Theorem~\ref{thm:exhaustive_strong}  yields the following consistency result.
\begin{theorem}\label{thm:theta_rate} For each $\gamma \in [0,\infty)$, and given a sample $\bf{X}$ generated by a (not necessarily mixing) stationary ergodic process $\mu \in \fE$, with probability $1$ the test $\langle f_t\rangle_{t \in \N}$ given by \eqref{eq:general_test_se_gamma} has the property that
  \begin{align*}
    &\lim_{t\to \infty}f_t(\bf X)= \begin{cases} +1 & \text{if} \ \mu \in \fC_{\gamma} \\
      -1 & \text{if} \ \mu \in \fE \setminus \fC_{\gamma}~. \end{cases}
  \end{align*}
\end{theorem}
Similarly, an argument based on Theorem~\ref{thm:whp_b},  leads to the following analogous result.  
\begin{theorem}\label{thm:theta_rate_b} For each $\gamma \in [0,\infty)$, and given a sample $\bf{X}$ generated by a (not necessarily mixing) stationary ergodic process $\mu \in \fE$, with probability $1$ the test $\langle \wt{f}_t\rangle$ given by \eqref{eq:general_test_se_gamma_b} has the property that
  \begin{align*} &\lim_{t\to \infty}\wt{f}_t(\bf X)= \begin{cases} +1 & \text{if} \ \mu \in \bfC_{\gamma} \\
      -1 & \text{if} \ \mu \in \fE \setminus \bfC_{\gamma}~. \end{cases}
  \end{align*}
\end{theorem}

Using Theorem~\ref{thm:theta_rate}, we can readily obtain a consistent test for independence on $\fE$. 
Consider a process $\bf X$ with distribution $\mu$. Recall that $\I$ denotes the set of all $\X$-valued i.i.d. processes, and observe that if $\mu \in \I$ then clearly $\norm{\boldsymbol{\alpha}}=0$. Thus, if $\mu \in \fE$ we can use \eqref{eq:general_test_se_gamma} with $\gamma=0$ 
to  test the null hypothesis $H_0$ that $\mu \in \I$ against the alternative hypothesis $H_1$ that $\mu \in \fE \setminus \I$ so that the process is stationary ergodic but it is not i.i.d..  
More specifically, let $\langle \zeta_t \rangle$ be a decreasing sequence of positive real numbers as before, such that $\lim_{t\rightarrow\infty} \zeta_t=0$. For a given process ${\bf X}$ and each $t \in \N$ define 
\begin{equation}\label{eq:ftt}
\ol{f}_t(\bf{X}):=\begin{cases} 
+1& \text{if} \ \wh{\alpha}_t({\bf X}) \leq \zeta_t\\
-1&\text{otherwise}
\end{cases}
\end{equation}
Observing that $\norm{\boldsymbol{\alpha}}=0$ for all i.i.d. processes $\mu \in \I$, the consistency of \eqref{eq:ftt} trivially follows from Theorem~\ref{thm:theta_rate}. Thus, we recover, albeit via a different approach, the main result of Morvai and Weiss \cite{MOR11} as the following corollary. 
\begin{corollary}[Morvai and Weiss \cite{MOR11}] \label{cor:iidt2} 
There exists a strongly consistent test for independence on $\fE$. 
That is, for a stationary ergodic process $\bf{X}$ with process distribution $\mu \in \fE$ with probability one it holds that
\begin{align*}
&\lim_{t\to \infty}\ol{f}_t(\bf X)=
                 \begin{cases} +1 & \text{if} \ \mu \in \fE \cap \I\\
                   -1 & \text{if} \ \mu \in \fE \setminus \I ~. \end{cases} 
\end{align*}
\end{corollary}

\end{document}